\documentclass[11pt,reqno]{amsart}
\usepackage{amsmath, amssymb, amsthm}
\usepackage{url}
\usepackage[breaklinks]{hyperref}

 \renewcommand{\a}{\alpha}
\renewcommand{\b}{\beta}

\newcommand{\g}{\gamma}
\newcommand{\G}{\Gamma}
\renewcommand{\l}{\lambda}
\renewcommand{\L}{\Lambda}

\renewcommand{\(}{\left\(}
\renewcommand{\)}{\right\)}
\renewcommand{\[}{\left\[}
\renewcommand{\]}{\right\]}
\let\dotlessi=\i
\numberwithin{equation}{section}
 \theoremstyle{plain}
\newtheorem{theorem}{Theorem}[section]
\newtheorem{lemma}[theorem]{Lemma}
\newtheorem{remark}[]{Remark}

\newtheorem{corollary}[theorem]{Corollary}

   \makeatletter
\def\proof{\@ifnextchar[{\@oproof}{\@nproof}}
\def\@oproof[#1][#2]{\trivlist\item[\hskip\labelsep\textit{#2 Proof of\
#1.}~]\ignorespaces}
\def\@nproof{\trivlist\item[\hskip\labelsep\textit{Proof.}~]\ignorespaces}

\makeatother

\begin{document}
\title[A Ramanujan-type formula for $\zeta^{2}(2m+1)$ and its generalizations]{A Ramanujan-type formula for $\zeta^{2}(2m+1)$ and its generalizations} 

\author{Atul Dixit and Rajat Gupta}
\address{Discipline of Mathematics, Indian Institute of Technology Gandhinagar, Palaj, Gandhinagar 382355, Gujarat, India} 
\email{adixit@iitgn.ac.in, rajat\_gupta@iitgn.ac.in}
\begin{abstract}
A Ramanujan-type formula involving the squares of odd zeta values is obtained. The crucial part in obtaining such a result is to conceive the correct analogue of the Eisenstein series involved in Ramanujan's formula for $\zeta(2m+1)$. The formula for $\zeta^{2}(2m+1)$ is then generalized in two different directions, one, by considering the generalized divisor function $\sigma_z(n)$, and the other, by studying a more general analogue of the aforementioned Eisenstein series, consisting of one more parameter $N$. A number of important special cases are derived from the first generalization. For example, we obtain a series representation for $\zeta(1+\omega)\zeta(-1-\omega)$, where $\omega$ is a non-trivial zero of $\zeta(z)$. We also evaluate a series involving the modified Bessel function of the second kind in the form of a rational linear combination of $\zeta(4k-1)$ and $\zeta(4k+1)$ for $k\in\mathbb{N}$. 
\end{abstract}
\maketitle

\section{Introduction}\label{intro}
The Riemann zeta function $\zeta(s)$ is one of the most important special functions of Mathematics. While the critical strip $0<$ Re$(s)<1$ is undoubtedly the most important region in the complex plane on account of the unsolved problem regarding the location of non-trivial zeros of $\zeta(s)$, namely, the Riemann Hypothesis, the right-half plane Re$(s)>1$ also has its own share of interesting unsolved problems to contribute to. For example, while it is known that all even zeta values $\zeta(2m)$, $m\in\mathbb{N}$, are transcendental, thanks to Euler's formula 
\begin{equation}\label{zetaevenint}
\zeta(2 m ) = (-1)^{m +1} \frac{(2\pi)^{2 m}B_{2 m }}{2 (2 m)!}
\end{equation}
and the facts that $\pi$ is transcendental and the Bernoulli numbers $B_m$ are rational, the arithmetic nature of the corresponding odd zeta values $\zeta(2m+1)$ is far from being known. So far the only explicit result in this direction is that of Ap\'{e}ry \cite{apery1}, \cite{apery2} which says $\zeta(3)$ is irrational. 

Though Rivoal \cite{rivoal}, and Ball and Rivoal \cite{ballrivoal} have shown that there are infinitely many odd zeta values that are irrational, one is unable to explicitly say which out of these (except $\zeta(3)$) is irrational. For any pair of positive integers $a$ and $b$, Haynes and Zudilin \cite[Theorem 1]{hayneszudilin} have shown that either there are infinitely many $m\in\mathbb{N}$ for which $\zeta(am+b)$ is irrational, or the sequence $\{q_m\}_{m=1}^{\infty}$ of common denominators of the rational elements of the set $\{\zeta(a+b), \zeta(2a+b), \cdots, \zeta(am+b)\}$ grows super-exponentially, i.e., $q_m^{1/m}\to\infty$ as $m\to\infty$. A beautiful result of Zudilin \cite{zudilin} states that at least one of the numbers $\zeta(5), \zeta(7), \zeta(9)$ and $\zeta(11)$ is irrational. A very recent result due to Rivoal and Zudilin \cite{rivoalzudilin} states that at least two of the numbers $\zeta(5), \zeta(7), \cdots, \zeta(69)$ are irrational.

One of the most important formulas for odd zeta values is that of Ramanujan \cite[p.~173, Ch. 14, Entry 21(i)]{ramnote}, namely, for $\a, \b>0$ with $\a\b=\pi^2$ and $m\in\mathbb{Z}, m\neq 0$,
\begin{align}\label{zetaodd}
\a^{-m}\left\{\frac{1}{2}\zeta(2m+1)+\sum_{n=1}^{\infty}\frac{n^{-2m-1}}{e^{2\a n}-1}\right\}&=(-\b)^{-m}\left\{\frac{1}{2}\zeta(2m+1)+\sum_{n=1}^{\infty}\frac{n^{-2m-1}}{e^{2\b n}-1}\right\}\nonumber\\
&\quad-2^{2m}\sum_{j=0}^{m+1}\frac{(-1)^jB_{2j}B_{2m+2-2j}}{(2j)!(2m+2-2j)!}\a^{m+1-j}\b^j.
\end{align}
This formula has number of applications. For example, as shown in \cite{gmr}, it encodes fundamental transformation properties of Eisenstein series on the full modular group and their Eichler integrals. See \cite{berndtstraubzeta} for more details. Identity \eqref{zetaodd} and its special cases, namely, for $\a,\b>0$ with $\a\b=\pi^2$, 
\begin{align}
\a^{m}\sum_{n=1}^{\infty}\frac{n^{2m-1}}{e^{2\a n}-1}-(-\b)^m\sum_{n=1}^{\infty}\frac{n^{2m-1}}{e^{2\b n}-1}&=\left(\a^m-(-\b)^m\right)\frac{B_{2m}}{4m}\hspace{5mm}(m>1),\label{mg1}\\
\a\sum_{n=1}^{\infty}\frac{n}{e^{2n\a}-1}+\b\sum_{n=1}^{\infty}\frac{n}{e^{2n\b}-1}&=\frac{\a+\b}{24}-\frac{1}{4},\label{m-1}\\
\sum_{n=1}^{\infty}\frac{1}{n(e^{2n\a}-1)}-\sum_{n=1}^{\infty}\frac{1}{n(e^{2n\b}-1)}&=\frac{\b-\a}{12}+\frac{1}{4}\log\left(\frac{\a}{\b}\right)\label{m0},
\end{align}
are known to have applications in theoretical computer science \cite{kirprod} in the analysis of special data structures and algorithms. 

The first published proof of \eqref{zetaodd} is due to Malurkar \cite{malurkar} although he was not aware that this formula can be found in Ramanujan's Notebooks. Grosswald too rediscovered this formula and studied it more generally in \cite{gross1}, \cite{gross2}. Berndt \cite[Theorem 2.2]{berndtrocky} derived a general formula from which both Euler's formula \eqref{zetaevenint} and Ramanujan's formula \eqref{zetaodd} follow as special cases, thus showing that Euler's and Ramanujan's formulas are natural companions of each other. For an up-to-date history and developments related to Ramanujan's formulas \eqref{zetaodd}-\eqref{m0}, we refer the reader to \cite{berndtstraubzeta}. Very recently O'Sullivan \cite[Theorem 1.3]{osullivan} has found non-holomorphic analogues of the formulas of Ramanujan, Grosswald and Berndt containing Eichler integrals of holomorphic Eisenstein series. 

One more special case of Ramanujan's formula, other than \eqref{mg1}-\eqref{m0}, can be obtained by letting $\a=\b=\pi$ and $m$ to be odd, thereby resulting in a formula of Lerch \cite{lerch}, namely,
\begin{align}\label{lerchh}
\zeta(2m+1)+2\sum_{n=1}^{\infty}\frac{1}{n^{2m+1}(e^{2\pi n}-1)}=\pi^{2m+1}2^{2m}\sum_{j=0}^{m+1}\frac{(-1)^{j+1}B_{2j}B_{2m+2-2j}}{(2j)!(2m+2-2j)!}.
\end{align}
This implies \cite{gmr} that when $m$ is odd, at least one of $\zeta(2m+1)$ and $\sum_{n=1}^{\infty}\frac{1}{n^{2m+1}(e^{2\pi n}-1)}$ is transcendental.

Suppose we now ask ourselves if formulas analogous to \eqref{zetaevenint} and \eqref{zetaodd} could be obtained for squares of the zeta values? One reason why one may want to look at this is to obtain more information on the arithmetic nature of $\zeta(2m+1)$; indeed, for a certain $k\in\mathbb{N}$, $\zeta(2k+1)$ would be irrational if $\zeta^{2}(2k+1)$ turns out to be so. 

It is easy to see that the formula for $\zeta^{2}(2m)$ is trivially obtained by squaring both sides of \eqref{zetaevenint}. However, if one squares both sides of \eqref{zetaodd}, the resulting formula for $\zeta^{2}(2m+1)$ is very cumbersome with no hopes of further simplification or of any use. 
It is important to mention here a famous quote of G.~H.~Hardy \cite[p.~85]{apology}: `\emph{Beauty is the first test: there is no permanent place in the world for ugly mathematics}'.

The question that remains then is, does there exist a formula for $\zeta^{2}(2m+1)$ which almost matches \eqref{zetaodd} in terms of elegance? Not only do we affirmatively answer this question in this paper, but we also generalize our result in two different directions. In order to derive such a result, however, it is important to understand the work of Koshliakov in \cite{koshlondon}, and also in \cite{koshliakov2}, which builds the foundation of our work.

Note that Ramanujan's formula for $\zeta(2m+1)$ involves the Lambert series
\begin{equation*}
\sum_{n=1}^{\infty}\frac{a(n)e^{-2\pi nx}}{1-e^{-2\pi nx}}=\sum_{n=1}^{\infty}\frac{a(n)}{e^{2\pi nx}-1}
\end{equation*}
with $a(n)=n^{-2m-1}$. For $m<0$, this Lambert series is essentially the Eisenstein series of weight $-2m$ (except for the constant term), whereas for $m\geq 0$, it can be regarded as a ``negative weight Eisenstein series''. In his first Notebook, Ramanujan has few other results involving the negative weight Eisenstein series. The reader is referred to an interesting article of Duke \cite{duke} in this regard.

The function $1/(e^{2\pi x}-1)$ has simple poles at $x=0, \pm in, n\in\mathbb{N}$, and hence the partial fraction decomposition
\begin{equation}\label{pf}
\frac{1}{e^{2\pi x}-1}=-\frac{1}{2}+\frac{1}{2\pi x}+\frac{x}{\pi}\sum_{n=1}^{\infty}\frac{1}{x^2+n^2},
\end{equation}
so that the pole at $x=\pm in$ has residue $\frac{1}{2\pi}$.

Koshliakov \cite{koshlondon} studied a function which has simple pole at $x=\pm in$ at each $n\in\mathbb{N}$, analogous to $1/(e^{2\pi x}-1)$, but with residue $\frac{1}{2\pi}d(n)$, where $d(n)$ denotes the number of divisors of $n$. This function is given by
\begin{equation}\label{minusspl}
\Omega(x):=2\sum_{j=1}^{\infty}d(j)\left(K_{0}\left(4\pi\epsilon\sqrt{jx}\right)+K_{0}\left(4\pi\overline{\epsilon}\sqrt{jx}\right)\right).
\end{equation}
Here $K_{z}(x):=\frac{\pi}{2}\frac{\left(I_{-z}(x)-I_{z}(x)\right)}{\sin z\pi}$ is the modified Bessel function of the second kind of order $z$ \cite[p.~78]{watson-1944a} with $I_{z}(x)$ being that of the first kind \cite[p.~77]{watson-1944a}. Also, here, and throughout the paper, $\epsilon = \exp(\frac{i\pi}{4})$ so that $\overline{\epsilon}= \exp(-\frac{i\pi}{4})$. It satisfies a relation \cite[Equation 5]{koshlyakov-1929a}, \cite[Equation 7]{koshlondon} analogous to \eqref{pf}, namely,
\begin{equation*}
\Omega(x) = - \gamma - \frac{1}{2} \log x - \frac{1}{4 \pi x} +
\frac{x}{\pi} \sum_{j=1}^{\infty} \frac{d(j)}{x^{2} + j^{2}},
\end{equation*}
where $\gamma$ denotes Euler's constant. The above formula implies, in particular, that $\Omega(x)$ is real for $x>0$. The function $\Omega(x)$ plays an instrumental role in Koshliakov's extremely clever proof in \cite{koshlyakov-1929a} of the Vorono\"{\dotlessi} summation formula for $d(n)$ and satisfies many beautiful properties, for example, for $c=$Re$(s)>1$ \cite[Equation (11)]{koshlondon},
\begin{equation*}
\Omega(x)=\frac{1}{2\pi i}\int_{(c)}\frac{\zeta^{2}(1-s)x^{-s}}{2\cos\left(\frac{1}{2}\pi s\right)}\, ds,
\end{equation*}
where here, and throughout the paper, $\int_{(c)}$ denotes the line integral $\int_{c-i\infty}^{c+i\infty}$.

In the same paper \cite[Equations (27), (29)]{koshlondon}, Koshliakov gave two beautiful closed-form evaluations of infinite series involving the function $\Omega(x)$, namely, for $m>0$,
\begin{equation}\label{hsok1}
\sum_{n=1}^{\infty}n^{4m+1}d(n)\Omega(n)=\frac{B_{4m+2}^{2}}{(4m+2)^2}\left\{\log(2\pi)-\sum_{k=1}^{4m+1}\frac{1}{k}-\frac{\zeta'(4m+2)}{\zeta(4m+2)}\right\},
\end{equation}
and\footnote{In \cite[Equation (48)]{koshliakov2}, Koshliakov incorrectly evaluated this series as
\begin{equation*}
\sum_{n=1}^{\infty}nd(n)\Omega(n)=\frac{1}{4}\left\{\log(2\pi)-\frac{9}{8}-\frac{6}{\pi^2}\zeta'(2)\right\}.
\end{equation*}}
\begin{equation}\label{hsok2}
\sum_{n=1}^{\infty}nd(n)\Omega(n)=\frac{1}{144}\left\{\log(2\pi)-1-\frac{6}{\pi^2}\zeta'(2)\right\}-\frac{1}{32\pi}.
\end{equation}
Observing the analogy between $1/(e^{2\pi x}-1)$ and $\Omega(x)$, one can deduce that the first of these results is analogous to the one by Glaisher \cite{glaisher1889}, namely,
\begin{equation}\label{glashier}
\sum_{n=1}^{\infty}\frac{n^{4m+1}}{e^{2\pi n}-1}=\frac{B_{4m+2}}{2(4m+2)}\hspace{5mm}(m>0).
\end{equation}
The latter can actually be obtained by replacing $m$ by $-2m-1$, and letting $\a=\b=\pi$ in Ramanujan's formula \eqref{zetaodd}. Also, \eqref{hsok2} is analogous to Schl\"{o}milch's formula \cite{schlomilch}
\begin{equation*}
\sum_{n=1}^{\infty}\frac{n}{e^{2\pi n}-1}=\frac{1}{24}-\frac{1}{8\pi},
\end{equation*}
which follows from \eqref{m-1} by letting $\a=\b=\pi$. In \cite{koshliakov2}, Koshliakov studied a more general series than $\Omega(x)$ and analogous results.
\section{New results}\label{mresults}
Note that the series in \eqref{hsok1} can be constructed from \eqref{glashier} by replacing $1$ and $\frac{1}{e^{2\pi n}-1}$ in the latter by $d(n)$ and $\Omega(n)$ respectively. This analogy and the aforementioned discussion suggests that a full-fledged Ramanujan-type formula for $\zeta^{2}(2m+1)$ is not inconceivable. Indeed, we derive this formula in the theorem below.
\begin{theorem}\label{zetasquared}
For $\rho>0$, define $\Omega_{\rho}(n)$ be defined by
\begin{equation*}
\Omega_{\rho}(x):=2\sum_{j=1}^{\infty} d(j)\left(K_{0}(4\rho\epsilon\sqrt{jx})+K_{0}(4\rho\overline{\epsilon}\sqrt{jx})\right).
\end{equation*}
Let $m$ be a non-zero integer. For any $\alpha, \beta >0$ satisfying $\alpha \beta = \pi^2$,
\begin{align}\label{zetasquaredeqn}
&(\alpha^2)^{-m} \Bigg \{ \zeta^{2}(2m+1)\left(\gamma + \log\left(\frac{\alpha}{\pi}\right)- \frac{\zeta'(2m+1)}{\zeta(2m+1)}\right)+\sum_{n=1}^{\infty} \frac{d(n)\Omega_{\alpha}(n)}{n^{2m+1}} \Bigg \}\nonumber\\
&=(-\beta^2)^{-m} \Bigg \{ \zeta^{2}(2m+1)\left(\gamma + \log\left(\frac{\beta}{\pi}\right) - \frac{\zeta'(2m+1)}{\zeta(2m+1)}\right)+\sum_{n=1}^{\infty} \frac{d(n)\Omega_{\beta}(n)}{n^{2m+1}} \Bigg \} \nonumber\\
& \quad - \pi 2^{4m}\sum_{j=0}^{m+1}\frac{(-1)^{j}B^2_{2j}B^2_{2m+2-2j}}{((2j)!)^2((2m+2-2j)!)^2}(\alpha^{2})^{j}
(\beta^{2})^{m+1-j}.
\end{align}
\end{theorem}
\noindent
\begin{remark}\label{conv}
Note that while the infinite series in \eqref{zetaodd} can be rephrased as
\begin{equation*}
\sum_{n=1}^{\infty}\frac{n^{-2m-1}}{e^{2\a n}-1}=\sum_{n, j=1}^{\infty}n^{-2m-1}e^{-2jn\a}=\sum_{k=1}^{\infty}\sigma_{-(2m+1)}(k)e^{-2k\a},
\end{equation*}
where $\sigma_{w}(n):=\sum_{d|n}d^{w}$, the corresponding double series $\sum_{n=1}^{\infty} \frac{d(n)\Omega_{\alpha}(n)}{n^{2m+1}}$ in \eqref{zetasquaredeqn} can also be represented in the form of a single series:
\begin{align*}
\sum_{n=1}^{\infty} \frac{d(n)\Omega_{\alpha}(n)}{n^{2m+1}}&=2\sum_{n, j=1}^{\infty}n^{-2m-1}d(n)d(j)\left(K_{0}(4\a\epsilon\sqrt{jn})+K_{0}(4\a\overline{\epsilon}\sqrt{jn})\right)\\
&=2\sum_{k=1}^{\infty}b_m(k)\left(K_{0}(4\a\epsilon\sqrt{k})+K_{0}(4\a\overline{\epsilon}\sqrt{k})\right),
\end{align*}
where $b_m(k):=\sum_{n|k}n^{-2m-1}d(n)d(k/n)$. Note that for $m\geq 0$ and any $\delta>0$,
\begin{align}
b_m(k)\leq \sum_{n|k}d(n)d(k/n)<\!\!<k^{\frac{\delta}{2}}d(k)<\!\!<k^{\delta},
\end{align}
using repeatedly the fact \cite[p.~343, Theorem 315]{hardywright} that $d(n)=O\left(n^{\delta}\right)$ for any $\delta>0$. Similarly for $m<0$, 
\begin{align*}
b_m(k)\leq k^{-2m-1}\sum_{n|k}d(n)d(k/n)<\!\!<k^{-2m-1+\delta}.
\end{align*}
Thus, for $\delta>0$, \textup{Re}$(\a)>0$, $m\in\mathbb{Z}$ and $\tau=\max\{-2m-1,0\}$,
\begin{align*}
\left|\sum_{n=1}^{\infty} \frac{d(n)\Omega_{\alpha}(n)}{n^{2m+1}}\right|&<\!\!<\sum_{k=1}^{\infty}k^{\tau+\delta}\left(\left|K_{0}(4\a\epsilon\sqrt{k})\right|+\left|K_{0}(4\a\overline{\epsilon}\sqrt{k})\right|\right)\nonumber\\
&<\!\!<\sum_{k=1}^{\infty}k^{\tau+\delta-\frac{1}{4}}\exp{\left(-2\a\sqrt{2k}\right)},
\end{align*}
where in the last step we used the asymptotic formula \cite[p.~920, \textbf{8.451.6}]{gr}
\begin{equation*}
K_z(w)\sim\sqrt{\frac{\pi}{2w}}e^{-w}
\end{equation*}
as $w\to\infty$. Moreover, using \eqref{imtg} below with $z=0$ and $\ell>\max\{-2m,1\}, m\in\mathbb{Z}$, and using \eqref{strivert} and the fact that $\zeta(s)=O(1)$ for \textup{Re}$(s)>1$, it can be seen that $|\Omega_{\rho}(x)|<\!\!<_{\rho}x^{-\ell}$, which implies that the series $\sum_{n=1}^{\infty} \frac{d(n)\Omega_{\alpha}(n)}{n^{2m+1}}$ converges absolutely for every integer $m$.
\end{remark}

Let $\a=\b=\pi$ and $m$ odd in Theorem \ref{zetasquared}, and note from \eqref{minusspl} that $\Omega_{\pi}(n)=\Omega(n)$. This gives an analogue of Lerch's formula \eqref{lerchh}:
\begin{align}\label{lerchhh}
\zeta^{2}(2m+1)\left(\g-\frac{\zeta'(2m+1)}{\zeta(2m+1)}\right)+\sum_{n=1}^{\infty}\frac{d(n)\Omega(n)}{n^{2m+1}}=\pi^{4m+3}2^{4m-1}\sum_{j=0}^{m+1}\frac{(-1)^{j+1}B^2_{2j}B^2_{2m+2-2j}}{((2j)!)^2((2m+2-2j)!)^2}.
\end{align}
Actually we will derive two different generalizations of Theorem \ref{zetasquared} and then derive Theorem \ref{zetasquared} as their corollary.

The first generalization is concerned with the more general function $\Omega_{\rho}(x, z)$ defined by
\begin{align}\label{x,z}
\Omega_{\rho}(x,z):=2\sum_{j=1}^{\infty} \sigma_{-z}(j)j^{\frac{z}{2}}\left(e^{\pi i z/4}K_{z}\left(4 \rho \epsilon\sqrt{j x}\right)+e^{-\pi i z/4}K_{z}\left(4 \rho\overline{\epsilon} \sqrt{j x}\right)\right),
\end{align}
so that $\Omega_{\rho}(x,0)=\Omega_{\rho}(x)$. The special case $\Omega_{\pi}(x,z)$, which we will denote by $\Omega(x, z)$, was introduced in \cite[Equation (6.5)]{dixitmoll} and was instrumental in obtaining a simple proof of the Vorono\"{\dotlessi} summation formula associated with $\sigma_z(n)$ for analytic functions \cite[Section 6]{bdrz1}.

Our first generalization of Theorem \ref{zetasquared} is given in the following theorem.
\begin{theorem}\label{Main Theorem}
Let $m$ be a non-zero integer. Let $\alpha, \beta>0$ and $\alpha \beta =\pi^2$. Let $\mathfrak{B}:=\{0,\pm 2m\}\cup\{2\ell+1\}_{\ell=-\infty}^{\infty}$. Then for $z \in \mathbb{C}\backslash\mathfrak{B}$,
{\allowdisplaybreaks\begin{align}\label{maintheoremeqn}
&(\alpha^2)^{-m}\Bigg\{\frac{1}{2}\zeta(2m+1)\left(\Big( \frac{\alpha}{\pi} \Big)^{z}\zeta(2m+1-z)\zeta(1+z)+\Big( \frac{\alpha}{\pi} \Big)^{-z}\frac{\zeta(2m+1+z)\zeta(1-z)}{\cos\Big(\frac{\pi z}{2}\Big)}\right)\nonumber\\
&\qquad \qquad+\sum_{n=1}^{\infty}\frac{\sigma_{-z}(n)n^{z/2}\Omega_{\alpha}(n,z)}{n^{2m+1}} \Bigg\}\nonumber\\
&=(-\beta^2)^{-m}\Bigg\{\frac{1}{2}\zeta(2m+1)\left(\Big( \frac{\beta}{\pi} \Big)^{-z}\zeta(2m+1+z)\zeta(1-z)+\Big( \frac{\beta}{\pi} \Big)^{z}\frac{\zeta(2m+1-z)\zeta(1+z)}{\cos\Big(\frac{\pi z}{2}\Big)}\right)\nonumber\\
&\qquad \qquad\qquad+\sum_{n=1}^{\infty}\frac{\sigma_{z}(n)n^{-z/2}\Omega_{\beta}(n,-z)}{n^{2m+1}}\Bigg\}\nonumber\\
& \quad  + (-1)^m\pi 2^{2m}\sum_{j=0}^{m+1}\frac{(-1)^{j}B_{2j}B_{2m+2-2j}\hspace{1mm}\zeta(2m+2-2j-z)\zeta(2j+z)}{(2j)!(2m+2-2j)!} \Big( \frac{\alpha}{\beta}\Big)^{-1-m+2j+z/2}.
\end{align}}
\end{theorem}
\begin{remark}
The absolute convergence of the infinite series occurring in the above theorem can be proved in the same way as explained in Remark \ref{conv}.
\end{remark}
\begin{remark}
Note that in Theorem \ref{Main Theorem}, we cannot take $z$ to be an odd integer, say, $2\ell+1,-\infty<\ell<\infty$. However, the limiting case $z\to2\ell+1$ for integer values of $\ell$, seen relative to $m$ occurring in the above theorem, has been considered in Corollaries \ref{cor2ell+1}, \ref{cor2m-1,4m-1} and \ref{alphabeta-1}, with one omitted case discussed in Remark \ref{avoided}. Similarly, the limiting cases $z\to 0$ and $z\to\pm 2m, m\neq 0$ are dealt with in Theorem \ref{zetasquared} and Corollary \ref{2m-onecor} respectively.
\end{remark}
\noindent
It is easy to see that the above transformation is invariant if we simultaneously replace $\a$ by $\b$ and $z$ by $-z$. 

There are numerous corollaries that follow from the above theorem. These are given in the Section \ref{withz}. However, we highlight a few of them here. We begin by defining two auxiliary functions
\begin{align}\label{l1}
\L^{\pm}(x, z)&:=2\left(e^{\frac{i\pi z}{4}}\pm e^{-\frac{i\pi z}{4}}\right)\sum_{j=1}^{\infty}\sigma_{-z}(j)j^{\frac{z}{2}}\left(K_{z}\left(4\pi\epsilon\sqrt{jx}\right)\pm K_{z}\left(4\pi\overline{\epsilon}\sqrt{jx}\right)\right).
\end{align}
Since $K_{-z}(w)=K_z(w)$, it is easy to see that
\begin{align*}
\L^{\pm}(x, z)&=\Omega(x, z)\pm\Omega(x,-z).
\end{align*}
A result that we infer from Theorem \ref{Main Theorem} is now given.
\begin{theorem}\label{zetazetathm}
For $z\neq 0, \pm 2, 2\ell+1$, where $-\infty<\ell<\infty$,
\begin{align}\label{zetazeta}
\sum_{n=1}^{\infty}\sigma_{-z}(n)n^{1+\frac{z}{2}}\L^{+}(n, z)&=\frac{1}{24}\left(1+\sec\left(\tfrac{\pi z}{2}\right)\right)\left(\zeta(-1-z)\zeta(1+z)+\zeta(-1+z)\zeta(1-z)\right)\nonumber\\
&\quad-\frac{1}{4\pi}\zeta(-z)\zeta(z).
\end{align}
\end{theorem}
\begin{remark}
Let $z$ be a real number. On account of the Schwarz reflection principle, we have all of the zeta values on the right-hand side of the above equation to be real quantities. Hence, in this case, the series $\sum_{n=1}^{\infty}\sigma_{-z}(n)n^{1+\frac{z}{2}}\L^{+}(n, z)$ is real.
\end{remark}
The above theorem gives an interesting result when we specialize $z$ to be a zero of the Riemann zeta function.
\begin{corollary}\label{zeros}
\textup{(i)} Let $\omega$ denote a non-trivial zero of $\zeta(z)$. Then
\begin{align*}
\sum_{n=1}^{\infty}\sigma_{-\omega}(n)n^{1+\frac{\omega}{2}}\L^{+}(n, \omega)=\frac{1}{24}\left(1+\sec\left(\tfrac{\pi \omega}{2}\right)\right)\zeta(-1-\omega)\zeta(1+\omega).
\end{align*}
\textup{(ii)} For $k\in\mathbb{N}$,
\begin{align*}
\sum_{n=1}^{\infty}\sigma_{4k}(n)n^{1-2k}\L^{+}(n, -4k)=-\frac{1}{12}\left(\frac{B_{4k}}{4k}\zeta(4k-1)+\frac{B_{4k+2}}{4k+2}\zeta(4k+1)\right).
\end{align*}
\end{corollary}
A special case of Part (ii) of Corollary \ref{zeros}, in turn, gives the following important corollary:
\begin{corollary}\label{aperyapp}
We have,
\begin{align}\label{aperyappeqn}
\sum_{n=1}^{\infty}\frac{\sigma_{4}(n)}{n}\L^{+}(n,-4)=\frac{1}{12}\left(\frac{1}{120}\zeta(3)-\frac{1}{252}\zeta(5)\right).
\end{align}
Thus, at least one of the quantities $\zeta(5)$ and $\sum_{n=1}^{\infty}\frac{\sigma_{4}(n)}{n}\L^{+}(n,-4)$ is irrational.
\end{corollary}
It is widely believed \cite[Conjecture 27]{waldschmidt} that for any $n  \in \mathbb{N}$, and any non-zero polynomial $P \in \mathbb{Q}[x_0, x_1, \cdots, x_n]$, $P(\pi, \zeta(3), \zeta(5), \cdots, \zeta(2n+1)) \neq 0 $, that is, $\pi$ and all odd zeta values are algebraically independent over $\mathbb{Q}$. In particular, if proven true, it would imply that $\zeta(3)$ and $\zeta(5)$ are algebraically independent over $\mathbb{Q}$. In light of this, it is highly probable that the series in \eqref{aperyappeqn} is transcendental, and hence irrational. 
However, the conjecture on algebraic independence of the odd zeta values has not yet been proved even in the case of finitely many odd zeta values, or, even just $\zeta(3)$ and $\zeta(5)$. Hence it would be phenomenal if the series in \eqref{aperyappeqn} turns out to be rational, for then it would prove that $\zeta(5)$ is irrational. However, it seems unlikely that the series would be rational.

As will be shown in Section \ref{withz}, the limiting case $z\to \pm 1$ of Theorem \ref{zetazetathm} given below links four important constants, namely, $\pi$, Euler's constant $\g$, $\zeta(3)$ and the Glaisher-Kinkelin constant $A$ defined by \cite[p.~39, Equation (2)]{srichoi}:
\begin{equation}\label{gkc}
\log(A)=\lim_{n\to\infty}\left\{\sum_{k=1}^{n}k\log(k)-\left(\frac{n^2}{2}+\frac{n}{2}+\frac{1}{12}\right)\log(n)+\frac{n^2}{4}\right\}.
\end{equation}
\begin{corollary}\label{gk}
Let $\sigma(n):=\sum_{d|n}d$. Then
\begin{align}\label{four}
\sum_{n=1}^{\infty}\sigma(n)\sqrt{n}\hspace{1mm}\L^{+}(n, 1)=\frac{6+6\g+3\pi-72\log(A)-\zeta(3)}{288\pi}.
\end{align}
\end{corollary}
Another generalization of \eqref{four}, different from Theorem \ref{zetazetathm}, is given in Corollary \ref{alphabeta-1} of Section \ref{withz}. Next, we give analogues of \eqref{mg1}-\eqref{m0}.
\begin{theorem}\label{mg1sqcor}
For a natural number $m>1, \a>0, \b>0$ and $\a\b=\pi^2$,
\begin{align}\label{mg1sq}
&(\alpha^2)^{m}\sum_{n=1}^{\infty}n^{2m-1} d(n)\Omega_{\alpha}(n)-(-\beta^2)^{m}\sum_{n=1}^{\infty}n^{2m-1} d(n)\Omega_{\beta}(n)\nonumber\\
&=-\frac{B^2_{2m}}{4m^2}\left\{(\alpha^2)^{m}\left(\log\left(\frac{\a}{2\pi^2}\right)+\sum_{k=1}^{2m-1}\frac{1}{k}+\frac{\zeta'(2m)}{\zeta(2m)}\right)-(-\beta^2)^{m}\left(\log\left(\frac{\b}{2\pi^2}\right)+\sum_{k=1}^{2m-1}\frac{1}{k}+\frac{\zeta'(2m)}{\zeta(2m)}\right)\right\}.
\end{align}
\end{theorem}
Letting $\a=\b=\pi$, replacing $m$ by $2m+1$ in \eqref{mg1sq} gives Koshliakov's identity \eqref{hsok1}.

Also, the $m=-1$ case of Theorem \ref{zetasquared} gives an analogue of \eqref{m-1}:
\begin{theorem}\label{mg2sq}
Let $A$ be defined in \eqref{gkc}. For $\alpha, \beta>0, \a\b = \pi^2$ we have,
\begin{align*}
&\alpha^2\sum_{n=1}^{\infty}n d(n)\Omega_{\alpha}(n)+\beta^2\sum_{n=1}^{\infty}n d(n)\Omega_{\beta}(n)\\
&=-\frac{\pi}{16}-\frac{1}{144}\left\{\a^2\left(\g+\log\left(\frac{\a}{\pi}\right)+1-12\log(A)\right)+\b^2\left(\g+\log\left(\frac{\b}{\pi}\right)+1-12\log(A)\right)\right\}.
\end{align*}
\end{theorem}
The Stieltjes constants $\g_n$ are defined by
\begin{equation}\label{stieltjesc}
\gamma_n= \lim_{j \to \infty} \bigg\{\sum_{k=1}^{j}\frac{(\log {k})^n}{k}-\frac{(\log {j})^{n+1}}{n+1} \bigg\}.
\end{equation}
The Dedekind eta function $\eta(w)$, defined by $\eta(w):=e^{2\pi iw/24}\prod_{n=1}^{\infty}(1-e^{2\pi inw}), \textup{Im}(w)>0$, satisfies a transformation formula \cite[p.~52, Theorem 3.4]{apostol2} under the general modular transformation $V(w)=(aw+b)/(cw+d), a, b, c, d\in\mathbb{Z}, ad-bc=1$. Let $\a\b=\pi^2$. When $V(w)=-1/w$, this transformation can be recast into
\begin{equation*}
\a^{1/4}e^{-\a/12}\prod_{n=1}^{\infty}(1-e^{-2\a n})=\b^{1/4}e^{-\b/12}\prod_{n=1}^{\infty}(1-e^{-2\b n}),
\end{equation*}
which is equivalent to \eqref{m0}. An analogue of \eqref{m0} is now given.
\begin{theorem}\label{dedez0}
For $\a, \b>0$ such that $\a\b=\pi^2$,
\begin{align}\label{dedez0eqn}
\sum_{n=1}^{\infty}&\frac{d(n)\Omega_{\a}(n)} {n}-\sum_{n=1}^{\infty} \frac{d(n)\Omega_{\b}(n)}{n}\nonumber\\
&=\frac{1}{144}\pi (\a^2-\b^2)+\frac{1}{48}\log\left(\frac{\a}{\b}\right)\left\{48\g^2+96\g_1-3\pi^2-4\log^{2}\left(\frac{\a}{\b}\right)\right\}.
\end{align}
\end{theorem}
\begin{remark}
The generalization of the above result with an extra variable $z$ is given in Corollary \ref{dedez}.
\end{remark}
Recently the first author and Maji \cite[Theorem 1.2]{dixitmaji1} generalized Ramanujan's formula \eqref{zetaodd} to obtain the following relation between any two odd zeta values of the form $\zeta(2m+1)$ and $\zeta(2Nm+1)$, where $N$ is an odd positive integer, $m\in\mathbb{Z}\backslash\{0\}$, and where $\a,\b>0$ with $\a\b^{N}=\pi^{N+1}$:
{\allowdisplaybreaks\begin{align}\label{zetageneqn}
&\a^{-\frac{2Nm}{N+1}}\left(\frac{1}{2}\zeta(2Nm+1)+\sum_{n=1}^{\infty}\frac{n^{-2Nm-1}}{\textup{exp}\left((2n)^{N}\a\right)-1}\right)\nonumber\\
&=\left(-\b^{\frac{2N}{N+1}}\right)^{-m}\frac{2^{2m(N-1)}}{N}\Bigg(\frac{1}{2}\zeta(2m+1)+(-1)^{\frac{N+3}{2}}\sum_{j=\frac{-(N-1)}{2}}^{\frac{N-1}{2}}(-1)^{j}\sum_{n=1}^{\infty}\frac{n^{-2m-1}}{\textup{exp}\left((2n)^{\frac{1}{N}}\b e^{\frac{i\pi j}{N}}\right)-1}\Bigg)\nonumber\\
&\quad+(-1)^{m+\frac{N+3}{2}}2^{2Nm}\sum_{j=0}^{\left\lfloor\frac{N+1}{2N}+m\right\rfloor}\frac{(-1)^jB_{2j}B_{N+1+2N(m-j)}}{(2j)!(N+1+2N(m-j))!}\a^{\frac{2j}{N+1}}\b^{N+\frac{2N^2(m-j)}{N+1}}.
\end{align}}%
On page $332$ of his Lost Notebook \cite{lnb}, Ramanujan embarked on obtaining some result on a more general form of the Lambert series occurring in the above equation, namely, $\sum_{n=1}^{\infty}n^{N-2h}/(\exp{(n^{N}x)}-1)$, however he does not give any result. See \cite{dixitmaji1} for details. 

Our next generalization of Theorem \ref{zetasquared} gives an analogue of \eqref{zetageneqn} by relating $\zeta^2(2m+1)$ with $\zeta^2(2Nm+1)$ .
\begin{theorem}\label{zetasquaredN}
Let $N$ be an odd positive integer. Let $\a,\b>0$ with $\a\b^{N}=\pi^{N+1}$. Then
\begin{align*}
&\left(\a^{\frac{4N}{N+1}}\right)^{-m}\left\{\zeta^{2}(2Nm+1)\left(\g+\log\left(\frac{\a}{\pi}\right)-N\frac{\zeta'(2Nm+1)}{\zeta(2Nm+1)}\right)+\sum_{n=1}^{\infty}\frac{d(n)}{n^{2Nm+1}}\Omega_{\a}\left(n^N\right)\right\}\nonumber\\
&=\frac{1}{N}\left(-\b^{\frac{4N}{N+1}}\right)^{-m}\bigg\{\zeta^{2}(2m+1)\left(\g+\frac{1}{N}\log\left(\frac{\b}{\pi}\right)-\frac{1}{N}\frac{\zeta'(2m+1)}{\zeta(2m+1)}\right)\nonumber\\
&\qquad\qquad\qquad\qquad\quad+(-1)^{\frac{N+3}{2}}\sum_{j=-\frac{(N-1)}{2}}^{\frac{N-1}{2}}(-1)^j\sum_{n=1}^{\infty}\frac{d(n)}{n^{2m+1}}\Omega_{\b}\left(e^{-\frac{ij\pi}{N}}n^{1/N}\right)\bigg\}\nonumber\\
&\quad-\pi 2^{2N-2+4Nm}\sum_{j=0}^{\left\lfloor\frac{N+1}{2N}+m\right\rfloor}\frac{(-1)^{j}2^{4j(1-N)}B_{2j}^{2}B_{N+1+2N(m-j)}^{2}}{(2j)!^2(N+1+2N(m-j))!^2}\a^{\frac{4j}{N+1}}\b^{2N+\frac{4N^2(m-j)}{N+1}}.
\end{align*}
\end{theorem}
As can be readily seen, letting $N=1$ in the above theorem gives Theorem \ref{zetasquared}. A counterpart of the above theorem for even $N$, which is an analogue of Wigert's formula \cite[Theorem 1.4]{dixitmaji1}, is given at the end of Section \ref{withN}.  

This paper is organized as follows. In Section \ref{prelim}, we collect the preliminary results often used in the proofs. Section \ref{withz} commences with a lemma used the proof of the main theorem following it, that is, Theorem \ref{Main Theorem}. The proof of Theorem \ref{Main Theorem} is then followed by several of its corollaries, namely, Theorem \ref{zetasquared}, the limiting case $z\to 2m$ discussed in Corollary \ref{2m-onecor}, the limiting case $z\to 2\ell+1$ discussed in Corollaries \ref{cor2ell+1}, \ref{cor2m-1,4m-1} and \ref{alphabeta-1}. A yet another special case of Theorem \ref{Main Theorem}, namely, Theorem \ref{zetazetathm}, is then proved, followed by the proofs of Corollaries \ref{zeros} and \ref{aperyapp}. Corollary \ref{gk} is then obtained from Corollary \ref{alphabeta-1}. An analogue of the transformation formula for the logarithm of the Dedekind eta function, that is, of \eqref{m0}, which involves an extra variable $z$, is then derived in Corollary \ref{dedez} with its special case Corollary \ref{zeros1} stated next. This is followed by the proofs of Theorems \ref{dedez0}, \ref{mg1sqcor} and \ref{mg2sq} in that order. Section \ref{withN} is devoted to the second generalization of Theorem \ref{zetasquared} consisting of the extra parameter $N$. This is achieved by first deriving a general result, Theorem \ref{analoguedixitmaji}, and then proving Theorem \ref{zetasquaredN} as its special case. The analogue of the latter theorem for $N$ even is then given in Theorem \ref{zetasquaredNeven}. The paper ends with some concluding remarks and directions for further research.

\section{Nuts and bolts}\label{prelim}
Here we collect well-known results that are frequently used in the sequel. Stirling's formula for the Gamma function on a vertical strip states that for $a\leq\sigma\leq b$ and $|t|\geq 1$,
\begin{equation}\label{strivert}
|\Gamma(\sigma+it)|=(2\pi)^{\tfrac{1}{2}}|t|^{\sigma-\tfrac{1}{2}}e^{-\tfrac{1}{2}\pi |t|}\left(1+O\left(\frac{1}{|t|}\right)\right).
\end{equation}
The functional equation of the Riemann zeta function is given by \cite[p.~13, Equation (2.1.1)]{titch}
\begin{equation}\label{zetafe}
\zeta(s)=2^s\pi^{s-1}\G(1-s)\zeta(1-s)\sin\left(\frac{\pi s}{2}\right).
\end{equation}
For Re$(s)>\max\{1,1+$Re$(b)\}$, we have \cite[p.~8, Equation (1.3.1)]{titch}
\begin{equation}\label{sigmadir}
\sum_{n=1}^{\infty}\frac{\sigma_{b}(n)}{n^s}=\zeta(s)\zeta(s-b).
\end{equation}
In 1885, Stieltjes found the Laurent series expansion of $\zeta(s)$ around $s=1$, which is given by \cite[Theorem 3.2.1]{lagariaseuler}
\begin{equation*}
\zeta(s)=\frac{1}{s-1}+\g+\sum_{n=1}^{\infty}\frac{(-1)^n\g_n}{n!}(s-1)^n,
\end{equation*}
where $\g_n$ is defined in \eqref{stieltjesc}. Hence, as $s\to 1$,
\begin{align}\label{zetsti}
\zeta(s)=\frac{1}{s-1}+\g-\g_1 (s-1)+O\left(|s-1|^2\right).
\end{align}
We will use the elementary fact $\sigma_{z}(n)n^{-z/2}=\sigma_{-z}(n)n^{z/2}$ without mention in the sequel.
\section{The first generalization of Theorem \ref{zetasquared} with an additional variable $z$}\label{withz}
Here we first prove Theorem \ref{Main Theorem} and then obtain several interesting corollaries from it, including Theorem \ref{zetasquared}. We begin with a lemma.

\begin{lemma}\label{omalxz}
Let $\Omega_{\rho}(x, z)$ be defined as in \eqref{x,z}. Then for $c=$\textup{Re}$(s)>1\pm$\textup{Re}$\left(\frac{z}{2}\right)$,
\begin{align*}
\Omega_{\rho}(x, z)=\frac{1}{2 \pi i}\int_{(c)}\frac{\zeta\left(1-s+\frac{z}{2} \right)\zeta\left(1-s-\frac{z}{2} \right)}{2\cos\left(\frac{\pi}{2}(s+\frac{z}{2} )\right)}\left(\frac{\rho^2}{\pi^2}x\right)^{-s}\, ds.
\end{align*}
\end{lemma}
\begin{proof}
For \textup{Re}$(s)>\pm$\textup{Re}$(z)$ and Re$(a)>0$, we know that \cite[p.~115, Formula 11.1]{ober}
\begin{align}\label{ana1}
\int_{0}^{\infty}t^{s-1}K_{z}(a t)\, dt= 2^{s-2}a^{-s}\Gamma\left(\frac{s-z}{2}\right)\Gamma\left(\frac{s+z}{2}\right).
\end{align}
Now employ the change of variable $t=\sqrt{x}$ and then replace $s$ by $2s$ so that for $c=$\textup{Re}$(s)>\pm$\textup{Re}$\left(\frac{z}{2}\right)$ and Re$(a)>0$,
\begin{align}\label{needed}
\int_{0}^{\infty}x^{s-1}K_{z}(a \sqrt{x})\, dx= 2^{2s-1}a^{-2s}\Gamma\left(s+\frac{z}{2}\right)\Gamma\left(s-\frac{z}{2}\right).
\end{align}
Replace $a$ by $4\rho\epsilon\sqrt{j}$ and then by $4\rho\overline{\epsilon}\sqrt{j}$ in \eqref{needed} and add the resulting two equations so that
\begin{align}\label{ana2}
&\int_{0}^{\infty}x^{s-1}(e^{\pi i z/4}K_{z}(4 \rho \epsilon \sqrt{j x} )+e^{-\pi i z/4}K_{z}(4 \rho \bar{\epsilon} \sqrt{j x}))\, dx\nonumber\\
&=\frac{1}{(2 \rho \sqrt{j})^{2s}}\Gamma\left(s+\frac{z}{2}\right)\Gamma\left(s-\frac{z}{2}\right)\cos\left(\frac{\pi}{2}\left(s-\frac{z}{2}\right)\right).
\end{align}
Employing the above identity in terms of its equivalent inverse Mellin transform representation in the definition of $\Omega_{\rho}(x, z)$, we deduce that
\begin{align*}
\Omega_{\rho}(x, z)&=\frac{2}{2\pi i}\sum_{j=1}^{\infty} \sigma_{-z}(j)j^{\frac{z}{2}}\int_{(c)}\frac{1}{(2 \rho \sqrt{j})^{2s}}\Gamma\left(s+\frac{z}{2}\right)\Gamma\left(s-\frac{z}{2}\right)\cos\left(\frac{\pi}{2}\left(s-\frac{z}{2}\right)\right)x^{-s}\, ds\nonumber\\
&=\frac{2}{2\pi i}\int_{(c)}\frac{1}{(2 \rho)^{2s}}\Gamma\left(s+\frac{z}{2}\right)\Gamma\left(s-\frac{z}{2}\right)\cos\left(\frac{\pi}{2}\left(s-\frac{z}{2}\right)\right)\left(\sum_{j=1}^{\infty}\frac{\sigma_{-z}(j)}{j^{s-\frac{z}{2}}}\right)x^{-s}\, ds,
\end{align*}
where the interchange of the order of summation and integration can be justified by means of absolute convergence which in turn follows from \eqref{strivert}. Using \eqref{sigmadir} with $s$ replaced by $s-\frac{z}{2}$ and $b$ by $-z$ in the above equation, we see that for $c=$\textup{Re}$(s)>1\pm$\textup{Re}$\left(\frac{z}{2}\right)$,
\begin{align}\label{imtg}
\Omega_{\rho}(x, z)&=\frac{2}{2\pi i}\int_{(c)}\zeta\left(s-\frac{z}{2} \right)\zeta\left(s+\frac{z}{2} \right)\Gamma\left(s+\frac{z}{2}\right)\Gamma\left(s-\frac{z}{2}\right)\cos\left(\frac{\pi}{2}\left(s-\frac{z}{2}\right)\right)\left(4\rho^2x\right)^{-s}ds\nonumber\\
&=\frac{1}{2 \pi i}\int_{(c)}\frac{\zeta\left(1-s+\frac{z}{2} \right)\zeta\left(1-s-\frac{z}{2} \right)}{2\cos\left(\frac{\pi}{2}\left(s+\frac{z}{2}\right)\right)}\left(\frac{\rho^2x}{\pi^2}\right)^{-s}ds,
\end{align}
where in the last step we used the functional equation \eqref{zetafe} twice, once, with $s$ replaced by $1-s+\frac{z}{2}$, and then with $1-s-\frac{z}{2}$.
\end{proof}
\begin{proof}[Theorem \textup{\ref{Main Theorem}}][]
We first assume Re$(z)\geq 0$ and let $m\in\mathbb{N}$. We begin with the series on the left-hand side of \eqref{maintheoremeqn}. Using the variant of the reflection formula for the Gamma function, namely, $\Gamma\left(\frac{1}{2}+w\right)\Gamma\left(\frac{1}{2}-w\right)=\frac{\pi}{\cos(\pi w)}, w\notin\mathbb{Z}-\frac{1}{2}$, followed by \eqref{strivert}, it can be shown that as $\textup{Im}(s)\to\infty$,
\begin{align}\label{cosbound}
\frac{1}{\left|\cos\left(\frac{\pi}{2}\left(s+\frac{z}{2}\right)\right)\right|}=2\exp{\left(-\tfrac{\pi}{2}\left|\textup{Im}(s)+\tfrac{1}{2}\textup{Im}(z)\right|\right)}\left(1+O\left(\frac{1}{|\textup{Im}(s)|}\right)\right).
\end{align}
This guarantees the interchange of the order of summation and integration in the second step below and therefore using Lemma \ref{omalxz}, we have 
\begin{align*}
\sum_{n=1}^{\infty} \frac{\sigma_{-z}(n) n^{z/2}\Omega_{\alpha}(n,z)}{n^{2m+1}}&=\sum_{n=1}^{\infty} \frac{\sigma_{-z}(n)}{n^{2m+1-\frac{z}{2}}}\frac{1}{2 \pi i}\int_{(c)}\frac{\zeta\left(1-s+\frac{z}{2} \right)\zeta\left(1-s-\frac{z}{2} \right)}{2\cos\left(\frac{\pi}{2}(s+\frac{z}{2} )\right)}\left(\frac{n\a^2}{\pi^2}\right)^{-s}\, ds\nonumber\\
&=\frac{1}{2 \pi i}\int_{(c)}\left(\sum_{n=1}^{\infty}\frac{\sigma_{-z}(n)}{n^{2m+1+s-\frac{z}{2}}}\right)\frac{\zeta\left(1-s+\frac{z}{2} \right)\zeta\left(1-s-\frac{z}{2} \right)}{2\cos\left(\frac{\pi}{2}(s+\frac{z}{2} )\right)}\left(\frac{\a^2}{\pi^2}\right)^{-s}\, ds.
\end{align*}
In order to represent the series $\sum_{n=1}^{\infty}\sigma_{-z}(n)n^{-2m-1-s+\frac{z}{2}}$ as the zeta product, we need $c >\max\{1+ \text{Re}(z/2), -2m \pm \text{Re}(z/2) \}=1+$Re$(z/2)$ since Re$(z)\geq 0$ and $m>0$. Thus for $c>1+$Re$(z/2)$,
\begin{align}\label{Parent sum 1}
\sum_{n=1}^{\infty} \frac{\sigma_{-z}(n) n^{z/2}\Omega_{\alpha}(n,z)}{n^{2m+1}}&=\frac{1}{2\pi i}\int_{(c)}F(s, z, m)\, ds,
\end{align}
where
\begin{align}\label{fszm}
F(s, z, m):=\frac{G(s, z, m)}{2\cos\Big(\frac{\pi}{2}(s+\frac{z}{2} )\Big)}\Bigg(\frac{\alpha^2}{\pi^2}\Bigg)^{-s},
\end{align}
with
\begin{align}\label{gszm}
G(s, z, m):=\zeta\Big(2m+1+s-z/2 \Big)\zeta\Big(2m+1+s+z/2 \Big)\zeta\Big(1-s+\frac{z}{2} \Big)\zeta\Big(1-s-\frac{z}{2} \Big).
\end{align}
Consider the contour $C$ determined by the line segments $[c - iT, c+ iT],[c + iT, -\lambda+ iT],[-\lambda + iT, -\lambda - iT],[-\lambda - iT, -c - iT],$ where $c=c'+\text{Re}(z/2)$ with $1<c'< 3$, and $\lambda = 2m + \text{Re}(z/2) + c'$ so that $-2m-3-\text{Re}(z/2)<-\lambda < -2m -1 -\text{Re}(z/2).$ In general, let $R_a$ denote the residue of the associated integrand, in this case $F(s, z, m)$, at $s=a$.

It is easy to see that the integrand $F(s, z, m)$ has simple poles at $-2m\pm z/2, \pm z/2$ due to each of the four zeta functions as well as simple poles at $-2k+1-z/2$ for $0\leq k\leq m+1$. The latter ones are due to the zeros of $\cos\left(\tfrac{\pi}{2}\left(s+\tfrac{z}{2}\right)\right)$ at $-2k+1-z/2$. Note that the zeros corresponding to $k>m+1$ get canceled by the corresponding zeros of $\zeta\left(2m+1+s+z/2 \right)$ at $s=-2j-2m-1-z/2, j>0$. The residues at the above poles can be easily calculated thus giving
\begin{align}\label{residues}
R_{-2m + z/2} &= \frac{1}{2}(-1)^m\left(\alpha/\pi\right)^{4m-z} \sec\left(\frac{\pi z}{2}\right)\zeta(1+2m)\zeta(1+z)\zeta(1+2m-z),\nonumber\\
R_{-2m - z/2} &= \frac{1}{2}(-1)^m\left(\frac{\alpha}{\pi}\right)^{4m+z} \zeta(1+2m)\zeta(1-z)\zeta(1+2m+z),\nonumber\\
R_{z/2} &= -\frac{1}{2}\left(\frac{\alpha}{\pi}\right)^{-z}\sec\left(\frac{\pi z}{2}\right)\zeta(1+2m)\zeta(1-z)\zeta(1+2m+z),\nonumber\\
R_{-z/2} &= -\frac{1}{2}\left(\frac{\alpha}{\pi}\right)^{z}\zeta(1+2m)\zeta(1+z)\zeta(1+2m-z),\nonumber\\
R_{1-2k-z/2}&= \frac{1}{4 \alpha}(-1)^{k+m}(2\pi)^{2+2m}\left(\frac{\pi}{\alpha}\right)^{1-4k-z}\frac{B_{2k}B_{2-2k+2m}\zeta(2k+z)\zeta(2-2k+2m-z)}{(2k)!(2-2k+2m)!}.
\end{align}
By Cauchy's residue theorem, 
\begin{align*}
&\frac{1}{2\pi i}\left[\int_{c-i T}^{c+I T}+\int_{c+i T}^{-\lambda +I T}+\int_{-\lambda + i T}^{-\lambda -I T}+\int_{-\lambda -i T}^{c-I T}\right]F(s, z, m)\, ds \nonumber\\
&= R_{z/2}+R_{-z/2}+R_{-2m-z/2}+R_{-2m+z/2}+\sum_{k=0}^{m+1}R_{1-2k-z/2}.
\end{align*}
Using \eqref{strivert} again and the elementary bounds on the Riemann zeta function, it can be easily shown that the integrals over the horizontal segments tend to zero at $T\to\infty$. Thus from \eqref{Parent sum 1},
\begin{align}\label{idtypri}
\sum_{n=1}^{\infty} \frac{\sigma_{-z}(n) n^{z/2}\Omega_{\alpha}(n,z)}{n^{2m+1}}&=\frac{1}{2\pi i}\int_{(-\l)}F(s, z, m)\, ds+R_{z/2}+R_{-z/2}+R_{-2m-z/2}\nonumber\\
&\quad+R_{-2m+z/2}+\sum_{k=0}^{m+1}R_{1-2k-z/2}.
\end{align}
To evaluate the integral in the above equation, employ the change of variable $s=-w-2m$ and note from \eqref{gszm} that $G(-2m-w, z, m)=G(w,z,m)$. Along with recalling the fact that $\l=2m+$Re$(z/2)+c', 1<c'<3$, this gives
\begin{align}\label{transformed_sum}
\frac{1}{2\pi i}\int_{(-\l)}F(s, z, m)\, ds&=(-1)^m\Bigg(\frac{\alpha}{\pi}\Bigg)^{4m}\frac{1}{2\pi i}\int_{(c'+\textup{Re}(z/2))}\frac{G(w, z, m)}{2\cos\Big(\frac{\pi}{2}(s-\frac{z}{2} )\Big)}\Bigg(\frac{\alpha^2}{\pi^2}\Bigg)^{w}\, dw\nonumber\\
&=(-1)^m\Bigg(\frac{\alpha}{\pi}\Bigg)^{4m}\frac{1}{2\pi i}\int_{(c'+\textup{Re}(z/2))}\frac{G(w, -z, m)}{2\cos\Big(\frac{\pi}{2}(s-\frac{z}{2} )\Big)}\Bigg(\frac{\left(\pi^4/\alpha^2\right)}{\pi^2}\Bigg)^{-w}\, dw\nonumber\\
&=(-1)^m\Bigg(\frac{\alpha}{\pi}\Bigg)^{4m}\sum_{n=1}^{\infty} \frac{\sigma_{z}(n) n^{-z/2}\Omega_{\frac{\pi^2}{\a}}(n,-z)}{n^{2m+1}},
\end{align}
as can be seen from \eqref{Parent sum 1} and \eqref{fszm}. Note that in the second step in the above calculation, we have used the fact that $G(w,z,m)=G(w,-z,m)$.

Therefore from \eqref{residues}, \eqref{idtypri} and \eqref{transformed_sum}, we obtain upon simplification
\begin{align*}
&\sum_{n=1}^{\infty}\frac{\sigma_{-z}(n)n^{z/2}\Omega_{\alpha}(n,z)}{n^{2m+1}} \\
&\quad+\frac{1}{2}\zeta(2m+1)\Bigg\{\Big( \frac{\alpha}{\pi} \Big)^{z}\zeta(1+2m-z)\zeta(1+z)+\Big( \frac{\alpha}{\pi} \Big)^{-z}\frac{\zeta(1+2m+z)\zeta(1-z)}{\cos\Big(\frac{\pi z}{2}\Big)}\Bigg\}\\
&=(-1)^m\Bigg(\frac{\alpha}{\pi}\Bigg)^{4m}\sum_{n=1}^{\infty}\frac{\sigma_{z}(n)n^{-z/2}\Omega_{\frac{\pi^2}{\a}}(n,-z)}{n^{2m+1}}\\
&\quad+\frac{(-1)^m\alpha^{4m}}{2\pi^{4m}}\zeta(2m+1)\Bigg\{\Big( \frac{\alpha}{\pi} \Big)^{z}\zeta(1+2m+z)\zeta(1-z)+\Big( \frac{\alpha}{\pi} \Big)^{-z}\frac{\zeta(1+2m-z)\zeta(1+z)}{\cos\Big(\frac{\pi z}{2}\Big)}\Bigg\}\\
&\quad + \sum_{k=0}^{m+1}\frac{(-1)^{m+k}}{4 \alpha}(2 \pi)^{2+2m}\Big( \frac{\pi}{\alpha}\Big)^{1-4k-z}\frac{B_{2k}B_{2-2k+2m}}{(2k)!(2-2k+2m)!}\zeta(2-2k+2m-z)\zeta(2k+z).
\end{align*}
Multiply both sides of the above equation by $(\a^2)^{-m}$, employ the fact $\a\b=\pi^2$, and observe that
\begin{equation*}
(\alpha^2)^{-m}\frac{1}{\alpha}\pi^{2+2m}\Big( \frac{\pi}{\alpha}\Big)^{1-4k-z}=\pi \Big( \frac{\alpha}{\beta}\Big)^{-1-m+2k+z/2}.
\end{equation*}
This leads to \eqref{maintheoremeqn} and thus completes the proof of Theorem \ref{Main Theorem} for Re$(z)\geq 0$ and $m>0$. For Re$(z)\leq 0$ and $m>0$, simply swap $\alpha$ and $\b$ and simultaneously replace $z$ by $-z$. This leaves \eqref{maintheoremeqn} invariant and we are back to the case just proved above. For $m<0$, the result can be proved exact along the similar lines as above.
\end{proof}
\begin{corollary}
Theorem \ref{zetasquared} is valid.
\end{corollary}
\begin{proof}
Let $z\to 0$ in Theorem \ref{Main Theorem}. Using \eqref{zetaevenint} twice and the fact that $\a\b=\pi^2$, it is easy to see that as $z\to 0$, the right-hand side of \eqref{maintheoremeqn} tends to that of \eqref{zetasquaredeqn}. Thus it suffices to show that
\begin{align}\label{limit}
&\lim_{z\to 0}\left(\Big(\frac{\alpha}{\pi} \Big)^{z}\zeta(2m+1-z)\zeta(1+z)+\Big( \frac{\alpha}{\pi} \Big)^{-z}\frac{\zeta(2m+1+z)\zeta(1-z)}{\cos\Big(\frac{\pi z}{2}\Big)}\right)\nonumber\\
&=2\zeta(2m+1)\left(\gamma + \log\left(\frac{\alpha}{\pi}\right)- \frac{\zeta'(2m+1)}{\zeta(2m+1)}\right).
\end{align}
To that end, we use \eqref{zetsti} with $s$ replaced by $1+z$ and two other well-known expansions as $z\to 0$:
\begin{align*}
\zeta(2m+1\pm z)&=\zeta(2m+1)\pm z\zeta'(2m+1)+O_{m}(|z|),\\
\left(\frac{\a}{\pi}\right)^{z}&=\exp{\left(z\log\left(\frac{\a}{\pi}\right)\right)}=1+z\log\left(\frac{\a}{\pi}\right)+O_{\a}\left(|z|^2\right).
\end{align*}
Thus, as $z\to 0$,
\begin{align}\label{1}
&\Big(\frac{\alpha}{\pi} \Big)^{z}\zeta(2m+1-z)\zeta(1+z)\nonumber\\
&=\left(1+z\log\left(\frac{\a}{\pi}\right)+O_{\a}\left(|z|^2\right)\right)\left(\zeta(2m+1)\pm z\zeta'(2m+1)+O_{m}(|z|)\right)\left(\frac{1}{z}+\g+O(|z|)\right)\nonumber\\ 
&=\frac{1}{z}\zeta(2m+1)+\zeta(2m+1)\left(\left(\gamma + \log\left(\frac{\alpha}{\pi}\right)\right)- \frac{\zeta'(2m+1)}{\zeta(2m+1)}\right)+O\left(|z|\right).
\end{align}
Now applying the functional equation \eqref{zetafe} with $s$ replaced by $1-z$ in the first step below and then using some of the above expansions as well as 
\begin{align*}
\G(z)&=\frac{1}{z}-\g+O(|z|),\\
\zeta(z)&=-\frac{1}{2}-\frac{z}{2}\log(2\pi)+O\left(|z|^2\right),
\end{align*}
as $z\to 0$, in the second step, we see that
\begin{align}\label{2}
&\Big( \frac{\alpha}{\pi} \Big)^{-z}\frac{\zeta(2m+1+z)\zeta(1-z)}{\cos\Big(\frac{\pi z}{2}\Big)}\nonumber\\
&=2(2\a)^{-z}\zeta(z)\G(z)\zeta(2m+1+z)\nonumber\\
&=2\left(1-z\log(2\a)+O_{\a}\left(|z|^2\right)\right)\left(\frac{1}{z}-\g+O(|z|)\right)\left(-\frac{1}{2}-\frac{z}{2}\log(2\pi)+O\left(|z|^2\right)\right)\nonumber\\
&\quad\times\left(\zeta(2m+1)+z\zeta'(2m+1)+O_{m}(|z|)\right)\nonumber\\
&=-\frac{1}{z}\zeta(2m+1)+\zeta(2m+1)\left(\left(\gamma + \log\left(\frac{\alpha}{\pi}\right)\right)- \frac{\zeta'(2m+1)}{\zeta(2m+1)}\right)+O\left(|z|\right).
\end{align}
Adding \eqref{1} and \eqref{2} and then letting $z\to 0$ leads to \eqref{limit}. This proves \eqref{zetasquaredeqn}.
\end{proof}

\begin{corollary}\label{2m-onecor}
Let $\a,\b>0$ such that $\a\b=\pi^2$. Then for $m\in\mathbb{N}$,
\begin{align}\label{2m-one}
&(\a^2)^{-m}\sum_{n=1}^{\infty}\frac{\sigma_{-2m}(n)\Omega_{\alpha}(n,2m)}{n^{m+1}}-(-\b^2)^{-m}\sum_{n=1}^{\infty}\frac{\sigma_{-2m}(n)\Omega_{\beta}(n,-2m)}{n^{m+1}}\nonumber\\
&=\frac{\zeta^{2}(2m+1)}{2\pi^{2m}}\log\left(\frac{\a}{\b}\right)+\frac{1}{2}(-1)^{m+1}\pi^{2m}\left(\a^{-4m}-\b^{-4m}\right)\zeta(2m+1)\zeta(4m+1)\zeta(1-2m)\nonumber\\
&\quad+\begin{cases}
\frac{1}{6}\pi^{2m+3}2^{4m-1}\frac{B_{2m}B_{2m+2}}{(2m)!(2m+2)!}\left(\frac{\a}{\b}-\frac{\b}{\a}\right), \text{if}\hspace{1mm}m>0,\\
0,\hspace{55mm}\text{if}\hspace{1mm}m\leq-1.\\
\end{cases}
\end{align}
\end{corollary}
\begin{proof}
Write \eqref{maintheoremeqn} in the form
\begin{align}\label{maintheoremeqn1}
&(\alpha^2)^{-m}\sum_{n=1}^{\infty}\frac{\sigma_{-z}(n)n^{z/2}\Omega_{\alpha}(n,z)}{n^{2m+1}}-(-\beta^2)^{-m}\sum_{n=1}^{\infty}\frac{\sigma_{z}(n)n^{-z/2}\Omega_{\beta}(n,-z)}{n^{2m+1}}\nonumber\\
&=\frac{1}{2}\zeta(2m+1)\bigg\{\pi^z\zeta(2m+1+z)\zeta(1-z)\left((-1)^m\beta^{-z-2m}-\a^{-z-2m}/\cos\left(\tfrac{\pi z}{2}\right)\right)\nonumber\\
&\qquad\qquad\qquad\quad+\pi^{-z}\zeta(2m+1-z)\zeta(1+z)\left((-1)^m\b^{z-2m}/\cos\left(\tfrac{\pi z}{2}\right)-\a^{z-2m}\right)\bigg\}\nonumber\\
& \quad  + (-1)^m\pi 2^{2m}\sum_{j=0}^{m+1}\frac{(-1)^{j}B_{2j}B_{2m+2-2j}\hspace{1mm}\zeta(2m+2-2j-z)\zeta(2j+z)}{(2j)!(2m+2-2j)!} \Big( \frac{\alpha}{\beta}\Big)^{-1-m+2j+z/2}.
\end{align}
Now let $z\to 2m$. It is easy to see that the left-hand side of the above equation then becomes that of \eqref{2m-one}. On the right-hand side, we have
\begin{align}\label{limit1}
&\lim_{z\to 2m}\pi^z\zeta(2m+1+z)\zeta(1-z)\left((-1)^m\beta^{-z-2m}-\a^{-z-2m}/\cos\left(\tfrac{\pi z}{2}\right)\right)\nonumber\\
&=\frac{1}{2}(-1)^{m+1}\pi^{2m}\left(\a^{-4m}-\b^{-4m}\right)\zeta(2m+1)\zeta(4m+1)\zeta(1-2m).
\end{align}
Also, using \eqref{zetsti} with $s$ replaced by $2m+1-z$, we see that as $z\to 2m$,
\begin{align*}
\zeta(2m+1-z)=\frac{1}{2m-z}+\g+O\left(|z-2m|\right).
\end{align*}
Hence
\begin{align}\label{limit2}
&\lim_{z\to 2m}\pi^{-z}\zeta(2m+1-z)\zeta(1+z)\left((-1)^m\b^{z-2m}/\cos\left(\tfrac{\pi z}{2}\right)-\a^{z-2m}\right)\nonumber\\
&=\pi^{-2m}\zeta(2m+1)\lim_{z\to 2m}\left(\frac{1}{2m-z}+\g+O\left(|z-2m|\right)\right)\left((-1)^m\b^{z-2m}/\cos\left(\tfrac{\pi z}{2}\right)-\a^{z-2m}\right)\nonumber\\
&=\pi^{-2m}\zeta(2m+1)\lim_{z\to 2m}\frac{1}{2m-z}\left((-1)^m\b^{z-2m}/\cos\left(\tfrac{\pi z}{2}\right)-\a^{z-2m}\right)\nonumber\\
&=\frac{1}{2\pi^{2m}}\zeta^{2}(2m+1)\log\left(\frac{\a}{\b}\right).
\end{align}
Note that in the finite sum on the right-hand side of \eqref{maintheoremeqn1}, only the terms corresponding to $j=0,1$ survive when $m>0$. On simplifying these two terms using \eqref{zetaevenint}, and then combining with \eqref{limit1} and \eqref{limit2}, one arrives at 
$\frac{1}{6}\pi^{2m+3}2^{4m-1}\frac{B_{2m}B_{2m+2}}{(2m)!(2m+2)!}\left(\frac{\a}{\b}-\frac{\b}{\a}\right)$.
Also, for $m\leq-1$, the sum is easily seen to be zero.
\end{proof}
\begin{remark}
We note that letting $\a=\b=\pi$ in Corollary \ref{2m-onecor} gives $0=0$.
\end{remark}
\begin{corollary}\label{cor2ell+1}
Let $m\in\mathbb{N}, \a,\b>0$ such that $\a\b=\pi^2$. Then,
\begin{align}\label{2ell+1}
&(\a^2)^{-m}\sum_{n=1}^{\infty}\frac{\sigma_{-2\ell-1}(n)\Omega_{\alpha}(n,2\ell+1)}{n^{2m-\ell+1/2}}-(-\b^2)^{-m}\sum_{n=1}^{\infty}\frac{\sigma_{-2\ell-1}(n)\Omega_{\beta}(n,-2\ell-1)}{n^{2m-\ell+1/2}}\nonumber\\
&=(-1)^m\pi 2^{2m}\sum_{j=0}^{m+1}\frac{(-1)^{j}B_{2j}B_{2m+2-2j}}{(2j)!(2m+2-2j)!}\zeta(2m-2\ell-2j+1)\zeta(2j+2\ell+1)\left(\frac{\a}{\b}\right)^{2j+\ell-m-\frac{1}{2}}\nonumber\\
&\quad+\pi(-1)^{\ell}2^{2m}\zeta(2m+1)\left(\frac{\b}{\pi}\right)^{2\ell-2m+1}\frac{(2\ell-2m)!B_{2\ell+2}}{(2\ell+2)!}\zeta(2\ell-2m+1)\nonumber\\
&\quad-\pi(-1)^{\ell}2^{2m}\zeta(2m+1)\begin{cases}
(-1)^{m}\left(\frac{\a}{\pi}\right)^{-2\ell-2m-1}\frac{(2\ell)!B_{2m+2\ell+2}}{(2m+2\ell+2)!}\zeta(2\ell+1),\text{if}\hspace{1mm}\ell\geq\max\{1,m+1\},\\
0, \text{if}\hspace{1mm}1\leq\ell<-m-1.
\end{cases}
\end{align}
\end{corollary}
\begin{proof}
First let $\ell\geq\max\{1,m+1\}$. Let $z\to 2\ell+1$ in Theorem \ref{Main Theorem}. Since $m\leq\ell-1$, we have $\zeta(2m-2\ell)=0$. Hence by L'Hopital's rule,
\begin{align}\label{zetalimi}
\lim_{z\to2\ell+1}\frac{\zeta(2m+1-z)}{\cos\left(\frac{\pi z}{2}\right)}&=\frac{2}{\pi}(-1)^{\ell}\zeta'(2m-2\ell)\nonumber\\
&=\frac{(-1)^m(2\ell-2m)!}{\pi(2\pi)^{2\ell-2m}}\zeta(2\ell-2m+1),
\end{align}
where in the last step, we used the well-known formula \cite[p.~167]{srichoi}
\begin{equation*}
\zeta'(-2k)=\frac{(-1)^k(2k)!}{2(2\pi)^{2k}}\zeta(2k+1).
\end{equation*}
We need to use \eqref{zetalimi} once with $m=0$ and once more with a general $m$. This gives \eqref{2ell+1} upon simplification.

Now let $1\leq\ell<-m-1$. Again let $z\to2\ell+1$ in Theorem \ref{Main Theorem}. Then $\zeta(2m+1-z)=0=\zeta(1-z)$ and also $\zeta(2m+1+z)=0$. Hence 
\begin{equation*}
\lim_{z\to2\ell+1}\frac{\zeta(2m+1+z)\zeta(1-z)}{\cos\left(\frac{\pi z}{2}\right)}=0.
\end{equation*}
This gives \eqref{2ell+1} upon simplification.
\end{proof}
\begin{corollary}\label{cor2m-1,4m-1}
For $\a,\b>0, \a\b=\pi^2$, $m\in\mathbb{N}$ and $m>1$,
\begin{align}\label{2m-1,4m-1}
&\a^{2m}\sum_{n=1}^{\infty}\sigma_{2m-1}(n)n^{m-1/2}\Omega_{\alpha}(n,2m-1)-(-\b^2)^{m}\sum_{n=1}^{\infty}\sigma_{2m-1}(n)n^{m-1/2}\Omega_{\beta}(n,1-2m)\nonumber\\
&=-\frac{B_{2m}}{2m}\left\{\frac{\a(2m-2)!\zeta(2m-1)}{2^{2m}}+\frac{(-1)^{m+1}\pi 2^{2m-1}B_{2m}(4m-2)!}{(2m)!}\left(\frac{\b}{2\pi}\right)^{4m-1}\zeta(4m-1)\right\}.
\end{align}
\end{corollary}
\begin{proof}
Let $m\in\mathbb{Z}, m<-1$, and then $z=-2m-1$ in Theorem \ref{Main Theorem}. Then 
\begin{align}\label{one-a}
\zeta(2m+1-z)=0=\zeta(1-z), \zeta(1+z)=\zeta(-2m)\hspace{1mm}\text{and}\hspace{1mm}\zeta(2m+1+z)=\zeta(0)=-1/2.
\end{align}
Also,
\begin{align}\label{two-a}
\lim_{z\to-2m-1}\frac{\zeta(1-z)}{\cos\left(\frac{\pi z}{2}\right)}=2^{2m+2}\pi^{2m+1}(-2m-2)!\zeta(-2m-1),
\end{align}
whereas replacing $\ell$ by $-m-1$ in \eqref{zetalimi} gives
\begin{align}\label{three-a}
\lim_{z\to-2m-1}\frac{\zeta(2m+1-z)}{\cos\left(\frac{\pi z}{2}\right)}=\frac{(-1)^m(-4m-2)!}{\pi(2\pi)^{-4m-2}}\zeta(-4m-1).
\end{align}
Substituting \eqref{one-a}, \eqref{two-a}, \eqref{three-a} in Theorem \ref{Main Theorem} and then replacing $m$ by $-m$, we find that for $m>1$, 
\begin{align*}
&\a^{2m}\sum_{n=1}^{\infty}\sigma_{2m-1}(n)n^{m-1/2}\Omega_{\alpha}(n,2m-1)-(-\b^2)^{m}\sum_{n=1}^{\infty}\sigma_{2m-1}(n)n^{m-1/2}\Omega_{\beta}(n,1-2m)\nonumber\\
&=\zeta(1-2m)\left(\a 2^{-2m}(2m-2)!\zeta(2m-1)+\pi^{1-2m}\left(\tfrac{\beta}{2\pi}\right)^{4m-1}(4m-2)!\zeta(4m-1)\zeta(2m)\right).
\end{align*}
Now \eqref{2m-1,4m-1} follows by using $\zeta(1-2m)=-B_{2m}/(2m)$ \cite[p.~266, Theorem 12.16]{apostol-1998a} and Euler's formula \eqref{zetaevenint}.
\end{proof}
For example, the case $\a=\b=\pi$ and $m=2$ of the above corollary gives 
\begin{align*}
\sum_{n=1}^{\infty}\sigma_3(n)n^{3/2}\L^{-}(n, 3)=\frac{1}{960\pi^3}\left(\zeta(3)+\frac{1}{2}\zeta(7)\right),
\end{align*}
where $\L^{-}(x, z)$ is defined in \eqref{l1}.

A yet another special case of Theorem \ref{Main Theorem} is given.
\begin{corollary}\label{alphabeta-1}
For $\a,\b>0, \a\b=\pi^2$,
\begin{align}\label{alphabeta-2}
&\a^2\sum_{n=1}^{\infty}\sigma(n)\sqrt{n}\Omega_{\a}(n,-1)+\b^2\sum_{n=1}^{\infty}\sigma(n)\sqrt{n}\Omega_{\b}(n,1)\nonumber\\
&=\frac{\pi\a}{96}-\frac{\a^3}{288\pi^2}\zeta(3)-\frac{\b}{48}\log\left(\frac{\a}{\b}\right)+\frac{\b}{48}\g+\frac{\b}{4}\zeta'(-1).
\end{align}
\end{corollary}
\begin{proof}
Let $m=-1$ and then let $z\to-1$ in Theorem \ref{Main Theorem}. Note that this requires evaluating
\begin{align}\label{elll}
L:=\lim_{z\to-1}\left\{\frac{\b^{z+2}}{24\pi^z}\frac{\zeta(-1-z)\zeta(1+z)}{\cos\left(\frac{\pi z}{2}\right)}-\frac{\pi}{4}\zeta(z)\zeta(-z)\left(\frac{\a}{\b}\right)^{z/2}\right\}.
\end{align}
As $z\to-1$,
\begin{align*}
\frac{1}{\cos\left(\frac{\pi z}{2}\right)}&=\frac{(2/\pi)}{z+1}+O\left(|z+1|\right),\nonumber\\
\zeta(-z)&=\frac{-1}{z+1}+\g+O\left(|z+1|\right).
\end{align*}
Hence substituting the above Laurent series expansions in \eqref{elll}, we have
\begin{align}\label{ellll}
L&=L_1+\frac{\g\pi}{48}\left(\frac{\a}{\b}\right)^{-1/2},
\end{align}
where
\begin{align*}
L_1&=\lim_{z\to-1}\frac{1}{(z+1)}\left\{\frac{\b^{z+2}}{12\pi^{z+1}}\zeta(-1-z)\zeta(1+z)+\frac{\pi}{4}\zeta(z)\left(\frac{\a}{\b}\right)^{\frac{z}{2}}\right\}\nonumber\\
&=\lim_{z\to-1}\left\{\frac{\b^{z+2}}{12\pi^{z+1}}\log\left(\frac{\b}{\pi}\right)\zeta(-1-z)\zeta(1+z)+\frac{\pi}{4}\zeta'(z)\left(\frac{\a}{\b}\right)^{\frac{z}{2}}+\frac{\pi}{8}\zeta(z)\left(\frac{\a}{\b}\right)^{\frac{z}{2}}\log\left(\frac{\a}{\b}\right)\right\}\nonumber\\
&=\frac{\b}{48}\log\left(\frac{\b}{\pi}\right)+\frac{\pi}{4}\zeta'(-1)\frac{\sqrt{\b}}{\sqrt{\a}}-\frac{\pi}{96}\frac{\sqrt{\b}}{\sqrt{\a}}\log\left(\frac{\a}{\b}\right),
\end{align*}
where in the second step we used L'Hopital's rule.

Now substitute the above value of $L_1$ in \eqref{ellll}, and repeatedly use the fact $\a\b=\pi^2$ to obtain 
\begin{equation}\label{leval}
L=-\frac{\b}{48}\log\left(\frac{\a}{\b}\right)+\frac{\b}{48}\g+\frac{\b}{4}\zeta'(-1).
\end{equation}
Substituting \eqref{leval} into the identity resulting by letting $z\to-1$ in Theorem \ref{Main Theorem} and simplifying leads to \eqref{alphabeta-2}.
\end{proof}
\begin{remark}\label{avoided}
A further limiting case of Theorem \ref{Main Theorem} can be obtained by letting $z\to 2\ell+1$ with $-m\leq\ell<m, m>0, \ell\neq-1$. We refrain from giving this identity since it is quite complicated.
\end{remark}

We now prove other special cases of Theorem \ref{Main Theorem}, some of which are stated in the introduction.
\begin{proof}[Theorem \textup{\ref{zetazetathm}}][]
Let $m=-1$ and $\a=\b=\pi$ in Theorem \ref{Main Theorem}, use $\zeta(-1)=-1/12$, and simplify. This gives \eqref{zetazeta} upon recalling \eqref{l1} and observing that
\begin{align*}
\sum_{n=1}^{\infty}\sigma_{-z}(n)n^{1+\frac{z}{2}}\Omega(n,z)+\sum_{n=1}^{\infty}\sigma_{z}(n)n^{1-\frac{z}{2}}\Omega(n,-z)
=\sum_{n=1}^{\infty}\sigma_{-z}(n)n^{1+\frac{z}{2}}\L^{+}(n, z).
\end{align*}
\end{proof}

\begin{proof}[Corollary \textup{\ref{zeros}}][]
To prove part (i), let $z=\omega$ in Theorem \ref{zetazetathm}. Note then that not only is $\zeta(\omega)=0$ but also $\zeta(1-\omega)=0$, as is the well-known fact, or, which can be easily seen through the functional equation \eqref{zetafe}. 

For part (ii), let $z=-4k, k\in\mathbb{N}$ in Theorem \ref{zetazetathm}, note that $\zeta(-4k)=0$ and use the fact $\zeta(-\l)=-B_{\l+1}/(\l+1)$ for $\l\in\mathbb{N}$ once with $\l=4k-1$ and then with $\l=4k+1$.
\end{proof}

\begin{proof}[Corollary \textup{\ref{aperyapp}}][]
Let $k=1$ in Part (ii) of Corollary \ref{zeros} to get \eqref{aperyappeqn}. Now if both $\zeta(5)$ and $\sum_{n=1}^{\infty}\frac{\sigma_{4}(n)}{n}\L^{+}(n,-4)$ were rational, that would imply that $\zeta(3)$ is rational too. But this contradicts Ap\'{e}ry's result that $\zeta(3)$ is irrational. Hence the proposed claim is true.
\end{proof}

\begin{proof}[Corollary \textup{\ref{gk}}][]
Let $\a=\b=\pi$ in Corollary \ref{alphabeta-1} and use the relation \cite[Equation (3.18)]{choi2007} 
\begin{equation}\label{choi}
\zeta'(-1)=\frac{1}{12}-\log(A),
\end{equation}
where $A$ is the Glaisher-Kinkelin constant.
\end{proof}
We next give an analogue of the transformation formula for the logarithm of the Dedekind eta function, that is, of \eqref{m0}, which consists of an extra variable $z$.
\begin{corollary}\label{dedez}
For $z\neq 0, \pm 1$ and $\a,\b>0$ with $\a\b=\pi^2$, we have
\begin{align}\label{dedezeqn}
&\sum_{n=1}^{\infty}\frac{\sigma_{-z}(n)n^{z/2}\Omega_{\a}(n,z)}{n}-\sum_{n=1}^{\infty}\frac{\sigma_{z}(n)n^{-z/2}\Omega_{\b}(n,-z)}{n}\nonumber\\
&=-\frac{\pi}{12}\left(\frac{\a}{\b}\right)^{\frac{z}{2}}\left\{\frac{\a}{\b}\zeta(-z)\zeta(2+z)-\frac{\b}{\a}\zeta(z)\zeta(2-z)\right\}+\frac{1}{2}\left(\frac{\a}{\b}\right)^{-\frac{z}{2}}\nonumber\\
&\quad\times\bigg[\frac{d}{dz}\left\{\left(\left(\frac{\a}{\b}\right)^{z}-\sec\left(\frac{\pi z}{2}\right)\right)\zeta(1-z)\zeta(1+z)\right\}+\log\left(\frac{\a}{\b}\right)\sec\left(\frac{\pi z}{2}\right)\zeta(1-z)\zeta(1+z)\bigg].
\end{align}
\end{corollary}
\begin{proof}
The proof is quite similar to that of Theorem \ref{Main Theorem} and hence we will be very brief. Assume Re$(z)\geq0$. For $1+$Re$(z/2)<c=$Re$(s)<3+$Re$(z/2)$,
\begin{align*}
&\sum_{n=1}^{\infty}\frac{\sigma_{-z}(n)n^{z/2}\Omega_{\a}(n,z)}{n}\nonumber\\
&=\frac{1}{2\pi i}\int_{(c)}\zeta\left(1-s+\frac{z}{2}\right)\zeta\left(1-s-\frac{z}{2}\right)\zeta\left(1+s-\frac{z}{2}\right)\zeta\left(1+s+\frac{z}{2}\right)\frac{(\a^2/\pi^2)^{-s}}{2\cos\left(\frac{\pi}{2}\left(s+\frac{z}{2}\right)\right)}\, ds.
\end{align*}
Now shift the line of integration to Re$(s)=-\l$, where $1+$Re$(z/2)<\l=$Re$(s)<3+$Re$(z/2)$ by constructing a rectangular contour, and consider the contribution of the residues at the simple poles of the integrand at $\pm1-z/2$ (due to $\cos\left(\frac{\pi}{2}\left(s+\frac{z}{2}\right)\right)$) and at the double poles at $\pm z/2$ (due to $\zeta(1\pm s\pm z/2)$), which, upon simplification, turns out to be the right-hand side of \eqref{dedezeqn}. The application of the residue theorem then leads to \eqref{dedezeqn}. 

For Re$(z)\leq0$, simply swap $\a$ and $\b$ and simultaneously replace $z$ by $-z$ in \eqref{dedezeqn}. This leaves \eqref{dedezeqn} invariant, which proves the result by reducing it to the previous case Re$(z)\geq 0$.
\end{proof}
We now give a corollary of the above result when $z$ is specialized to be a non-trivial zero of $\zeta(z)$ or a trivial zero of the form $-4k-2, k\in\mathbb{N}$. We omit the proof as it is easy.
\begin{corollary}\label{zeros1}
\textup{(i)} Let $\omega$ denote a non-trivial zero of $\zeta(z)$. Then
\begin{align*}
\sum_{n=1}^{\infty}\sigma_{-\omega}(n)n^{\frac{\omega}{2}-1}\L^{-}(n, \omega)=-\frac{\pi}{12}\zeta(-\omega)\zeta(2+\omega)-\frac{1}{2}\left(1-\sec\left(\tfrac{\pi \omega}{2}\right)\right)\zeta(1+\omega)\zeta'(1-\omega).
\end{align*}
\textup{(ii)} For $k\in\mathbb{N}$,
\begin{align*}
\sum_{n=1}^{\infty}\sigma_{4k+2}(n)n^{-2k-2}\L^{-}(n, -4k-2)=\zeta(4k+3)\zeta'(-4k-1)-\zeta'(4k+3)\zeta(-4k-1).
\end{align*}
\end{corollary}
We now prove Theorems \ref{mg1sqcor}-\ref{dedez0}. We begin with that of Theorem \ref{dedez0} first since it follows from the above corollary.
\begin{proof}[Theorem \textup{\ref{dedez0}}][]
Let $z\to 0$ in Corollary \ref{dedez}. The only thing to be shown is that the right-hand side of \eqref{dedezeqn} reduces, as $z\to 0$, to that of \eqref{dedez0eqn}. To that end, first note that
\begin{align}\label{ab144}
\lim_{z\to 0}\left(-\frac{\pi}{12}\left(\frac{\a}{\b}\right)^{\frac{z}{2}}\left\{\frac{\a}{\b}\zeta(-z)\zeta(2+z)-\frac{\b}{\a}\zeta(z)\zeta(2-z)\right\}\right)&=-\frac{\pi}{12}\left(\frac{\a}{\b}-\frac{\b}{\a}\right)\left(\frac{-1}{2}\right)\left(\frac{\pi^2}{6}\right)\nonumber\\
&=\frac{\pi}{144}\left(\a^2-\b^2\right),
\end{align}
since $\a\b=\pi^2$. As $z\to 0$,
\begin{align*}
\sec\left(\frac{\pi z}{2}\right)=1+\frac{\pi^2z^2}{8}+O\left(|z|^4\right).
\end{align*}
This, along with \eqref{zetsti}, implies that as $z\to 0$,
\begin{equation}\label{seczz}
\sec\left(\frac{\pi z}{2}\right)\zeta(1-z)\zeta(1+z)=-\frac{1}{z^2}+\left(\g^2+2\g_1-\frac{\pi^2}{8}\right)+O\left(|z|^2\right).
\end{equation}
Since 
\begin{equation*}
\left(\frac{\a}{\b}\right)^z=1+z\log\left(\frac{\a}{\b}\right)+\frac{z^2}{2!}\log^{2}\left(\frac{\a}{\b}\right)+\frac{z^3}{3!}\log^{3}\left(\frac{\a}{\b}\right)+O\left(|z|^4\right),
\end{equation*}
we have
\begin{align}\label{abzz}
\left(\frac{\a}{\b}\right)^z\zeta(1-z)\zeta(1+z)&=-\frac{1}{z^2}-\frac{\log(\a/\b)}{z}+\left(\g^2+2\g_1-\frac{1}{2}\log^{2}\left(\frac{\a}{\b}\right)\right)\nonumber\\
&\quad+\log\left(\frac{\a}{\b}\right)\left(\g^2+2\g_1-\frac{1}{6}\log^{2}\left(\frac{\a}{\b}\right)\right)z+O\left(|z|^2\right)
\end{align}
as $z\to 0$. Hence from \eqref{seczz} and \eqref{abzz},
\begin{align}\label{dsecabzz}
&\frac{d}{dz}\left\{\left(\left(\frac{\a}{\b}\right)^{z}-\sec\left(\frac{\pi z}{2}\right)\right)\zeta(1-z)\zeta(1+z)\right\}\nonumber\\
&=\frac{\log(\a/\b)}{z^2}+\log\left(\frac{\a}{\b}\right)\left(\g^2+2\g_1-\frac{1}{6}\log^{2}\left(\frac{\a}{\b}\right)\right)+O(|z|),
\end{align}
as $z\to 0$. Therefore \eqref{seczz} and \eqref{dsecabzz} imply
\begin{align*}
&\lim_{z\to 0}\frac{1}{2}\left(\frac{\a}{\b}\right)^{-\frac{z}{2}}\bigg[\frac{d}{dz}\left\{\left(\left(\frac{\a}{\b}\right)^{z}-\sec\left(\frac{\pi z}{2}\right)\right)\zeta(1-z)\zeta(1+z)\right\}\nonumber\\
&\qquad+\log\left(\frac{\a}{\b}\right)\sec\left(\frac{\pi z}{2}\right)\zeta(1-z)\zeta(1+z)\bigg]\nonumber\\
&=\frac{1}{48}\log\left(\frac{\a}{\b}\right)\left\{48\g^2+96\g_1-3\pi^2-4\log^{2}\left(\frac{\a}{\b}\right)\right\},
\end{align*}
which, when combined with \eqref{ab144}, results in the right side of \eqref{dedez0eqn}.
\end{proof}
\begin{proof}[Theorem \textup{\ref{mg1sqcor}}][]
In Theorem \ref{zetasquared}, replace $m$ by $-m$, where the new $m$ is greater than $1$. This renders the finite sum on the right to be zero. Then one uses $\zeta(1-2m)=-B_{2m}/(2m)$ and the result
\begin{align*}
\frac{\zeta'(1-2m)}{\zeta(1-2m)}&=\log(2\pi)-\frac{\G'(2m)}{\G(2m)}-\frac{\zeta'(2m)}{\zeta(2m)}\nonumber\\
&=\log(2\pi)+\g-\sum_{n=0}^{2m-2}\frac{1}{n+1}-\frac{\zeta'(2m)}{\zeta(2m)},
\end{align*}
which easily follows from the functional equation \eqref{zetafe} and the well-known relation \cite[p.~903, Formula \textbf{8.362.1}]{gr} 
\begin{equation}\label{psipsi}
\frac{\G'(w)}{\G(w)}=-\g-\sum_{n=0}^{\infty}\left(\frac{1}{w+n}-\frac{1}{n+1}\right).
\end{equation}
\end{proof}
\begin{proof}[Theorem \textup{\ref{mg2sq}}][]
Let $m=-1$ in Theorem \ref{zetasquared} and use \eqref{choi}.
\end{proof}
As remarked in Section \ref{mresults}, Koshliakov's first identity \eqref{hsok1} is an easy consequence of our Theorem \ref{mg1sqcor}. We end this section with a proof of his second identity, namely, \eqref{hsok2}. To see this, let $m=-1$ in \eqref{lerchhh} so as to get
\begin{equation}\label{ksec}
\frac{1}{144}\left(\g+12\zeta'(-1)\right)+\sum_{n=1}^{\infty}nd(n)\Omega(n)=-\frac{1}{32}.
\end{equation}
Next, differentiate both sides of \eqref{zetafe} with respect to $s$, then set $s=-1$ and use $\G'(2)=1-\g$ with the last equality resulting from \eqref{psipsi}. This gives
\begin{equation*}
\zeta'(-1)=\frac{1}{12}\left(1-\g-\log(2\pi)\right)+\frac{1}{2\pi^2}\zeta'(2).
\end{equation*}
Substituting the above equation in \eqref{ksec} gives \eqref{hsok2}.
\begin{remark}
Note that we could have also let $z\to 0$ in Theorem \ref{zetazetathm} so as to obtain \eqref{hsok2}.
\end{remark}


\section{The second generalization of Theorem \ref{zetasquared} with an additional parameter $N$}\label{withN}
We prove Theorem \ref{zetasquaredN} as a consequence of a more general theorem. The latter, given below, is an analogue of \cite[Theorem 1.1]{dixitmaji1}. 
\begin{theorem}\label{analoguedixitmaji}
Let $N\in\mathbb{N}$ and $h \in \mathbb{Z}$ such that $h \neq \frac{N+1}{2}$. Let $\a,\b>0$ such that $\a\b^{N}=\pi^{N+1}$. Then
\begin{equation}\label{pasa}
\sum_{n=1}^{\infty} n^{N-2h}d(n)\Omega_{\alpha}{(n^N)}= P(\a) + S(\a),
\end{equation}
where
{\allowdisplaybreaks\begin{align}\label{pa}
P(\alpha)&:=\zeta^2(2h-N)\left\{-\left(\gamma+ \log\left(\frac{\alpha}{\pi}\right)\right) +N\frac{\zeta'(2h-N)}{\zeta(2h-N)}\right\}+
\Big( \frac{\pi}{\alpha}\Big)^{\frac{2(N-2h+1)}{N}}\nonumber\\
&\quad\times \frac{1}{N^2}\zeta^2\left(\frac{2h-1}{N}\right)\csc\left(\frac{\pi(2h-1)}{2N}\right)\left\{\frac{\pi}{4}\cot\left(\frac{(2h-1)\pi}{2N}\right) + N\gamma - \log\left(\frac{\alpha}{\pi}\right) - \frac{\zeta'\left(\frac{2h-1}{N}\right)}{\zeta\left(\frac{2h-1}{N}\right)}\right\}\nonumber\\
&\quad+\sum_{k=0}^{\lfloor \frac{h}{N} \rfloor}(-1)^{k+1}\pi^{1-4k}\alpha^{4k-2}\zeta^2(2k)\zeta^2(2h-2Nk),
\end{align}}
and
\begin{align}\label{sodd}
S(\alpha):= \frac{(-1)^{h+1}}{N}\bigg(\frac{\pi}{\alpha}\bigg)^{\frac{2(N-2h+1)}{N}}\sum_{j=-\frac{(N-1)}{2}}^{\frac{N-1}{2}}e^{\frac{ij(2h-1)\pi}{N}}\sum_{n=1} ^{\infty}n^{\frac{1-2h}{N}}d(n)\Omega_{\b}\left(e^{-\frac{\pi i j}{N}}n^{\frac{1}{N}}\right)
\end{align}
for $N$ odd, and
\begin{align}\label{seven}
S(\alpha):= \frac{(-1)^{h+1}}{N}\bigg(\frac{\pi}{\alpha}\bigg)^{\frac{2(N-2h+1)}{N}}\sum_{j=-\frac{N}{2}}^{\frac{N}{2}-1}e^{\frac{i(2j+1)(2h-1)\pi}{2N}}\sum_{n=1} ^{\infty}n^{\frac{1-2h}{N}}d(n)\Omega_{\b}\left(e^{-\frac{\pi i (2j+1)}{2N}}n^{\frac{1}{N}}\right)
\end{align}
for $N$ even.
\end{theorem}
\begin{proof}
We prove the result only in the case $h>\frac{N}{2}, h\neq\frac{N+1}{2}$. The case $h\leq N/2$ can be proved in a similar way.

Letting $z=0$ in Lemma \ref{omalxz}, replacing $\rho$ and $x$ by $\a$ and $n^{N}$ respectively, and using the resulting expression in the first step below, we see that for $c=$Re$(s)>\max\left\{1,\frac{N-2h+1}{N}\right\}=1$,
\begin{align*}
\sum_{n=1}^{\infty} n^{N-2h}d(n)\Omega_{\alpha}{(n^N)}&=\sum_{n=1}^{\infty}\frac{1}{2\pi i}\int_{(c)}\frac{d(n)}{n^{Ns-N+2h}}\frac{\zeta^{2}(1-s)}{2\cos\left(\frac{\pi s}{2}\right)}\left(\frac{\a}{\pi}\right)^{-2s}\, ds\nonumber\\
&=\frac{1}{2\pi i}\int_{(c)} \frac{\zeta^2(1-s) \zeta^2(Ns-N+2h)}{2\cos(\frac{\pi s}{2})}\left(\frac{\a}{\pi}\right)^{-2s}\, ds,
\end{align*}
where in the last step we interchanged the order of summation and convergence which is justified by the absolute convergence since $c>1$. We now evaluate this integral by shifting the line of integration.

Consider the contour determined by the line segments $[c-iT,c+iT],[c+iT,-\l +iT],[-\l +iT,-\l -iT]$ and $[-\l-iT,c-iT]$, where $\l>\frac{2h}{N}-1$. The integrand has double order poles at $0$ and $\frac{N-2h+1}{N}$ due to $\zeta^{2}(1-s)$ and $\zeta^2(Ns-N+2h)$ respectively. Also, it has simple poles at $1$ and at negative odd integers (due to $\cos\left(\frac{\pi s}{2}\right)$). However, the poles $-(2\ell-1)$ for $\ell>\lfloor\frac{h}{N}\rfloor$ get canceled by the trivial zeros of $\zeta^2(Ns-N+2h)$ at negative even integers. To observe this, note that the trivial zeros of $\zeta(Ns-N+2h)$ are at $\frac{N-2h-2\ell}{N}$, where $\ell\in\mathbb{N}$. A pole of the form $-(2j-1)$ would cancel with a zero of the form $\frac{N-2h-2\ell}{N}$ only when $\ell=Nj-h$, which implies $j>h/N$. Thus, only the simple poles at $-1, -3, \cdots, -(2\lfloor\frac{h}{N}\rfloor-1)$ contribute to the integral. The residues at the poles $0, \frac{N-2h+1}{N}, 1$, and at $-(2j-1)$, where $1\leq j\leq\left\lfloor\frac{h}{N}\right\rfloor$, are given by
{\allowdisplaybreaks\begin{align}\label{residuesN}
R_{0}&=\zeta^2(2h-N)\left\{-\left(\gamma +\log(\tfrac{\a}{\pi})\right) + N\frac{\zeta'(2h-N)}{\zeta(2h-N)}\right\},\nonumber\\
R_1&=-\frac{\pi}{4\a^2}\zeta^{2}(2h),\nonumber\\
R_{\frac{N-2h+1}{N}}&= \left(\frac{\alpha}{\pi}\right)^{-\frac{2(N-2h+1)}{N}}\frac{\zeta^2\left(\tfrac{2h-1}{N}\right)}{N^2\sin\left(\tfrac{\pi(2h-1)}{2N}\right)}\left\{\tfrac{\pi}{4}\cot\left(\tfrac{(2h-1)\pi}{2N}\right) + N\gamma - \log\left(\tfrac{\alpha}{\pi}\right) - \frac{\zeta'\left(\frac{2h-1}{N}\right)}{\zeta\left(\tfrac{2h-1}{N}\right)}\right\},\nonumber\\
R_{-(2j-1)}&=(-1)^{j+1}\pi^{1-4j}\alpha^{4j-2}\zeta^2(2j)\zeta^2(2h-2Nj).
\end{align}}
By Cauchy's residue theorem,
{\allowdisplaybreaks\begin{align*}
&\frac{1}{2\pi i }\left[\int_{c-iT}^{c+iT}+ \int_{c+iT}^{-\lambda +iT} + \int_{-\lambda +iT}^{-\lambda -iT} +\int_{-\lambda -iT}^{c-iT}\right]\frac{\zeta^2(1-s) \zeta^2(Ns-N+2h)}{2\cos(\frac{\pi s}{2})}\left(\frac{\a}{\pi}\right)^{-2s}\, ds\nonumber\\
&= R_{0} + R_{1} +R_{\frac{1+N-2h}{N}}+\sum_{j=1}^{\lfloor \frac{h}{N} \rfloor }R_{-(2j-1)}.
\end{align*}}
From \eqref{cosbound} and the elementary bounds on the Riemann zeta function, it can be seen that as $T\to\infty$, the integrals along the horizontal line segments go to zero.
\begin{align}\label{tbp}
\frac{1}{2\pi i }\int_{(c)}\frac{\zeta^2(1-s) \zeta^2(Ns-N+2h)}{2\cos(\frac{\pi s}{2})}\left(\frac{\a}{\pi}\right)^{-2s}\, ds=R_{0} + R_{1} +R_{\frac{1+N-2h}{N}}+\sum_{j=1}^{\lfloor \frac{h}{N} \rfloor }R_{-(2j-1)}+J_{\a}(N,h),
\end{align}
where
\begin{equation}\label{janh}
J_{\a}(N,h):=\frac{1}{2\pi i }\int_{(-\l)}\frac{\zeta^2(1-s) \zeta^2(Ns-N+2h)}{2\cos(\frac{\pi s}{2})}\left(\frac{\a}{\pi}\right)^{-2s}\, ds.
\end{equation}
Replace $s$ by $1-s$ in \eqref{janh}, use the functional equation \eqref{zetafe} and simplify to obtain
\begin{align*}
J_{\a}(N,h)= \frac{2^{4h-1}\pi^{4h}}{2\pi i \alpha^2}\int_{(1+\lambda)}\frac{\zeta^2(s)\zeta^2(1+Ns-2h)\Gamma^2(1+Ns-2h)\sin^2(\frac{N \pi s}{2})}{2^{2Ns}\pi^{2s(N+1)}\alpha^{-2s}\sin(\frac{\pi s}{2})}ds.
\end{align*}
Employ the change of variable $s_1=1+Ns-2h$ so that
\begin{align}\label{janh1}
J_{\a}(N,h)&=\frac{2}{2\pi i N}\bigg(\frac{\pi}{\alpha}\bigg)^{\frac{2(N+1-2h)}{N}}\int_{(c_{1})}\left(\frac{(2\pi)^{2}\pi^{\frac{2}{N}}}{\alpha^{\frac{2}{N}}}\right)^{-s_{1}}\nonumber\\
&\quad\times\zeta^2\left(\frac{s_{1}+2h-1}{N}\right)\zeta^2(s_{1})\Gamma^2(s_{1})\frac{\sin^2\left(\frac{\pi(s_{1}+2h-1)}{2}\right)}{\sin\left(\frac{\pi(s_{1}+2h-1)}{2N}\right)}ds_{1}.
\end{align}
Now
\begin{align}\label{Chebyshev n^N}
\frac{\sin(Nz)}{\sin(z)}= \sum_{j=-(N-1)}^{N-1}{\vphantom{\sum}}'' e^{ijz},
\end{align} 
where $''$ implies summation over $j=-(N-1), -(N-3),\cdots, N-3, N-1$.
Observe that we can write
\begin{equation}\label{zetaexpanse}
\zeta^2(s_1)\zeta^2\left(\frac{s_1+2h-1}{N}\right)=\sum_{m,n=1}^{\infty}m^{-\frac{(2h-1)}{N}}d(m)d(n)\left(nm^{\frac{1}{N}}\right)^{-s_1},
\end{equation} 
for, $\l>\frac{2h}{N}-1$ implies $c_1=$Re$(s_1)=1+N(1+\l)-2h>1$, and $\frac{c_{1}+2h-1}{N} =$ Re$\left(\frac{s_{1}+2h-1}{N}\right)=1+ \lambda > \frac{2h}{N} >1 $ since $h>N/2$. Hence substituting \eqref{Chebyshev n^N} and \eqref{zetaexpanse} in \eqref{janh1}, and interchanging the order of integration and the double sum, we have upon simplification,
\begin{align}\label{janh2}
J_{\a}(N,h)&=\frac{2(-1)^{h+1}}{2\pi i N}\bigg(\frac{\pi}{\alpha}\bigg)^{\frac{2(N-2h+1)}{N}}\sum_{j=-(N-1)}^{N-1}{\vphantom{\sum}}''e^{\frac{ij(2h-1)\pi}{2N}}\sum_{m,n=1}^{\infty}\frac{d(m)d(n)}{m^{\frac{2h-1}{N}}}\int_{(c_{1})}\frac{\G^2(s_1)\cos\left(\frac{\pi s_1}{2}\right)}{(2\pi)^{2s_1}(nX_{m,N,j})^{s_1}}\, ds_1,
\end{align}
where
\begin{equation}\label{xmnj}
X_{m,N,j}:=\left(\frac{\pi}{\a}\right)^{\frac{2}{N}}e^{-\frac{ij\pi}{2N}}m^{\frac{1}{N}}.
\end{equation}
Now from \eqref{ana1} we have, for $d=$Re$(s)>0$ and Re$(a)>0$,
\begin{equation*}
\frac{1}{2\pi i}\int_{(d)}2^{s-2}a^{-s}\G^{2}\left(\frac{s}{2}\right)x^{-s}\, ds=K_{0}(ax).
\end{equation*}
By the principle of analytic continuation, both sides of the above equation are valid for Re$(x)>0$. Hence, for $d>0$, the identity derived from \eqref{ana2} by replacing $\rho$ by $\pi$, namely,
\begin{equation}\label{k0k0}
K_0\left(4\pi\epsilon\sqrt{nx}\right)+K_0\left(4\pi\bar{\epsilon}\sqrt{nx}\right)=\frac{1}{2\pi i}\int_{(d)}\frac{\G^{2}(s)\cos\left(\frac{\pi s}{2}\right)}{(2\pi)^{2s}(nx)^s}\, ds,
\end{equation}
is valid for $-\frac{\pi}{2}<\arg(x)<\frac{\pi}{2}$. Now from \eqref{xmnj}, 
\begin{equation*}
-\frac{\pi}{2}<-\frac{(N-1)\pi}{2N}\leq\arg(X_{m, N, j})=-\frac{j\pi}{2N}\leq \frac{(N-1)\pi}{2N}<\frac{\pi}{2}.
\end{equation*}
Therefore invoking \eqref{k0k0} with $x$ replaced by $X_{m, N, j}$, substituting the resultant in \eqref{janh2} and noting \eqref{minusspl}, we deduce that
\begin{align}\label{janh3}
J_{\a}(N,h)&=\frac{(-1)^{h+1}}{N}\bigg(\frac{\pi}{\alpha}\bigg)^{\frac{2(N-2h+1)}{N}}\sum_{j=-(N-1)}^{N-1}{\vphantom{\sum}}''e^{\frac{ij(2h-1)\pi}{2N}}\sum_{m=1}^{\infty}\frac{d(m)\Omega\left(X_{m, N, j}\right)}{m^{\frac{2h-1}{N}}}\nonumber\\
&=\frac{(-1)^{h+1}}{N}\bigg(\frac{\pi}{\alpha}\bigg)^{\frac{2(N-2h+1)}{N}}\sum_{j=-(N-1)}^{N-1}{\vphantom{\sum}}''e^{\frac{ij(2h-1)\pi}{2N}}\sum_{m=1}^{\infty}\frac{d(m)}{m^{\frac{2h-1}{N}}}\Omega_{\b}\left(e^{-\frac{ij\pi}{2N}}m^{\frac{1}{N}}\right),
\end{align}
where the last step follows from the fact $\Omega\left(X_{m, N, j}\right)=\Omega_{\b}\left(e^{-\frac{ij\pi}{2N}}m^{\frac{1}{N}}\right)$, since $\a\b^{N}=\pi^{N+1}$. Substituting \eqref{residuesN} and \eqref{janh3} in \eqref{tbp} and distinguishing two cases according to the parity of $N$, we finally arrive at \eqref{pasa}, with $P(\a)$ defined in \eqref{pa} and $S(\a)$ defined in \eqref{sodd} and \eqref{seven}. This completes the proof of Theorem \ref{analoguedixitmaji} for $h>N/2$.
\end{proof}
\begin{remark}
The result in the case when $h=\frac{N+1}{2}$ can also be obtained by arguing in the similar way as in the proof of Theorem \ref{analoguedixitmaji}, however, the residual terms are very complicated, which is why we refrain from explicitly giving it here.
\end{remark}
\begin{proof}[Theorem \textup{\ref{zetasquaredN}}][]
Let $h=\frac{N+1}{2}+Nm$, $m\in\mathbb{Z}\backslash\{0\}$, in Theorem \ref{analoguedixitmaji}. Multiply both sides of the resulting identity by $\a^{-\frac{4Nm}{N+1}}$. Substituting $\pi$ by $\a^{\frac{1}{N+1}}\b^{\frac{N}{N+1}}$, we see that
\begin{align*}
\a^{-\frac{4Nm}{N+1}}\left(\frac{\pi}{\a}\right)^{-4m}=\b^{-\frac{4Nm}{N+1}},
\end{align*}
and
\begin{align*}
&\sum_{k=0}^{\lfloor\frac{N+1}{2N}+m\rfloor}(-1)^{k+1}\pi^{1-4k}\alpha^{4k-2}\zeta^2(2k)\zeta^2(N+1+2N(m-k))\nonumber\\
&=-\pi 2^{2N-2+4Nm}\sum_{k=0}^{\left\lfloor\frac{N+1}{2N}+m\right\rfloor}\frac{(-1)^{k}2^{4k(1-N)}B_{2k}^{2}B_{N+1+2N(m-k)}^{2}}{(2k)!^2(N+1+2N(m-k))!^2}\a^{\frac{4k}{N+1}}\b^{2N+\frac{4N^2(m-k)}{N+1}},
\end{align*}
where we have also employed \eqref{zetaevenint} twice.
\end{proof}

Now we state the counterpart of Theorem \ref{zetasquaredN} for $N$ even. 
\begin{theorem}\label{zetasquaredNeven}
Let $N$ be an even positive integer. Let $\a,\b>0$ with $\a\b^{N}=\pi^{N+1}$. Then
\begin{align*}
&\a^{-\left(\frac{4Nm-2}{N+1}\right)}\left\{\zeta^{2}(2Nm)\left(\g+\log\left(\frac{\a}{\pi}\right)-N\frac{\zeta'(2Nm)}{\zeta(2Nm)}\right)+\sum_{n=1}^{\infty}\frac{d(n)}{n^{2Nm}}\Omega_{\a}\left(n^N\right)\right\}\nonumber\\
&=\frac{1}{N}\b^{-\left(\frac{4Nm-2}{N+1}\right)}\bigg\{\frac{(-1)^m}{\cos\left(\frac{\pi}{2N}\right)}\zeta^{2}\left(2m+1-\tfrac{1}{N}\right)\bigg(\g+\frac{1}{N}\log\left(\frac{\b}{\pi}\right)-\frac{1}{N}\frac{\zeta'\left(2m+1-\frac{1}{N}\right)}{\zeta\left(2m+1-\frac{1}{N}\right)}\nonumber\\
&\quad+\frac{\pi}{4N}\tan\left(\frac{\pi}{2N}\right)\bigg)+(-1)^{\frac{N}{2}}\sum_{j=-\frac{N}{2}}^{\frac{N}{2}-1}e^{-\frac{i(2j+1)\pi}{N}}\sum_{n=1}^{\infty}\frac{d(n)}{n^{2m+1-\frac{1}{N}}}\Omega_{\b}\left(e^{-\frac{i(2j+1)\pi}{2N}}n^{1/N}\right)\bigg\}\nonumber\\
&\quad-\pi 2^{2N+4Nm-4}\sum_{j=0}^{m}\frac{(-1)^{j}2^{4j(1-N)}B_{2j}^{2}B_{N+2N(m-j)}^{2}}{(2j)!^2(N+2N(m-j))!^2}\a^{\frac{4j}{N+1}}\b^{2N+\frac{4N^2(m-j)-2N}{N+1}}.
\end{align*}
\end{theorem}
\begin{proof}
Let $h=\frac{N}{2}+Nm$ in Theorem \ref{analoguedixitmaji}, multiply both sides of the resulting identity by $\a^{-\frac{(4Nm-2)}{N+1}}$, and simplify as in the proof of Theorem \ref{zetasquaredN}.
\end{proof}

\section{Concluding remarks}\label{cr}
The series in Theorems \ref{zetasquared} and \ref{Main Theorem} involving $\Omega_{\a}(n)$ and $\Omega_{\a}(n, z)$ respectively are analogues of the Eisenstein series on $\textup{SL}_{2}(\mathbb{Z})$ as the latter occur in Ramanujan's formula for $\zeta(2m+1)$, namely, \eqref{zetaodd}, whereas Theorems \ref{zetasquared} and \ref{Main Theorem} are analogues of \eqref{zetaodd} for $\zeta^{2}(2m+1)$. This suggests that these series involving $\Omega_{\a}(n)$ and $\Omega_{\a}(n, z)$ may have ramifications in the theory of modular forms. This will be explored in a future publication.

Ramanujan's formula \eqref{zetaodd} can be obtained by equating the coefficients of $w^n, n\geq 1$, on both sides of 
\begin{align}\label{cotcoth}
\frac{\pi}{2}\cot(\sqrt{w\a})\textup{coth}(\sqrt{w\b})=\frac{1}{2w}+\frac{1}{2}\log\left(\frac{\b}{\a}\right)+\sum_{m=1}^{\infty}\left\{\frac{m\a\textup{coth}(m\a)}{w+m^2\a}+\frac{m\b\textup{coth}(m\b)}{w-m^2\b}\right\},
\end{align}
which is how Ramanujan arrived at his result except that in his version of the above identity in \cite[p.~318, Formula (21)]{lnb}, he missed $\frac{1}{2}\log\left(\frac{\b}{\a}\right)$. Sitaramachandrarao \cite{sitaram} was the first to obtain the correct partial fraction decomposition, that is, \eqref{cotcoth}. See \cite{berndtstraubzeta} for more details. This raises a question: is it possible to derive Theorem \ref{zetasquared} or any of its generalizations, namely, Theorems \ref{Main Theorem} and \ref{zetasquaredN}, using Ramanujan's method of decomposing a certian function into its partial fractions? 

Another question that one may ask is, do there exist Ramanujan-type formulas for $\zeta^{k}(2m+1), k\geq 3$? This would first require finding the right analogues of the functions $1/(e^{2\pi x}-1)$ and $\Omega(x)$ that would help us to find, in turn, the right analogues of the Eisenstein series in \eqref{zetaodd} and of the series in \eqref{zetasquaredeqn} respectively. However, the formulas for higher $k$ may become unwieldy. 

It would be interesting to see if our results, especially Theorems \ref{zetasquared}, \ref{mg1sqcor}, \ref{mg2sq} and \ref{dedez0}, are applicable in the analysis of some special data structures and algorithms in theoretical computer science in a similar way as Kirschenhofer and Prodinger utilized \eqref{zetaodd} and its special cases \eqref{mg1}-\eqref{m0} in \cite{kirprod}.

It would be worthwhile to see if Theorem \ref{zetazetathm}, Corollary \ref{zeros}(i) and Corollary \ref{zeros1}(i) have some implications in analytic number theory, especially in the study of the non-trivial zeros of the Riemann zeta function. Similarly, it would be interesting to see applications of Corollary \ref{zeros}(ii) in transcendental number theory, especially in light of \eqref{aperyappeqn} and its consequence that at least one of the numbers $\zeta(5)$ and $\sum_{n=1}^{\infty}\frac{\sigma_{4}(n)}{n}\L^{+}(n,-4)$ is irrational.

\begin{center}
\textbf{Acknowledgements}
\end{center}
The authors sincerely thank Bibekananda Maji for carefully reading the manuscript and suggesting some changes.

\end{document}